\documentclass[a4paper,fleqn]{cas-dc}
\usepackage[authoryear,longnamesfirst]{natbib}
\usepackage{amsmath, amssymb}
\newtheorem{thm}{Theorem}
\newtheorem{cor}{Corollary}
\newtheorem{lem}{Lemma}
\newtheorem{rmk}{Remark}
\newtheorem{pro}{Proposition}
\newenvironment{proof}{{\bf Proof.}}{\hfill$\square$}
\newtheorem{ass}{Assumption}

\begin{document}
\let\WriteBookmarks\relax
\def\floatpagepagefraction{1}
\def\textpagefraction{.001}

\shorttitle{Discrete Time Backward Stochastic LQ Control }    

\shortauthors{L. Hu, Q. Meng \& M. Tang  }  

\title [mode = title]{Discrete Time Backward Stochastic Linear Quadratic Optimal Control }  

\tnotemark[1] 

\tnotetext[1]{Qingxin Meng was supported by the National Natural Science Foundation of China (No. 12271158) and the Key Projects of Natural Science Foundation of Zhejiang Province (No. LZ22A010005).} 

%

\author[1]{Ligui Hu}
\ead{19548222635@163.com}
\ead[url]{}

\credit{Conceptualization, Methodology, Writing – original draft}

\affiliation[1]{organization={Department of Mathematical Sciences, Huzhou Normal University},
                addressline={}, 
                city={Huzhou},
                postcode={313000}, 
                state={Zhejiang},
                country={China}}

\author[1]{Qingxin Meng}
\cormark[1]
\ead{mqx@zjhu.edu.cn}
\ead[URL]{}

\author[1]{Maoning Tang}
\ead{tmorning@zjhu.edu.cn}
\ead[url]{}

\credit{Conceptualization, Supervision, Validation, Writing – review}

\credit{Conceptualization, Supervision, Funding acquisition, Writing – review}

\cortext[1]{Corresponding author. E-mail: mqx@zjhu.edu.cn (Q. Meng)}



\begin{abstract}
This paper focuses on the discrete-time backward stochastic linear quadratic (BSLQ) optimal control problem with nonhomogeneous system terms and cost function cross terms. The terminal constraint of such systems distinguishes it from forward stochastic systems, posing unique challenges for analysis and solution. Within the Hilbert space framework, we first clarify the necessary and sufficient conditions for problem solvability, then introduce the backward stochastic system maximum principle to derive the Hamiltonian system characterizing the optimal control. After equivalent transformation of the original problem, we use the decoupling method to obtain the corresponding Riccati equation, and present the explicit state feedback expression of the optimal control and the analytical form of the value function. Finally, numerical examples verify the effectiveness and feasibility of the proposed method. The innovation lies in expanding the model generality: addressing the structural asymmetry issue in the Riccati equation with cross-term cost functions, we propose system equivalent transformation and decoupling techniques. Our theoretical results provide a new analytical framework for dynamic optimization problems such as financial portfolio optimization and risk management.
\end{abstract}




\maketitle
\section{Introduction}
Stochastic optimal control is a fundamental branch of modern control theory, concerned with designing strategies to optimize cost functionals for dynamical systems influenced by random disturbances. Linear-quadratic (LQ) control, originating from pioneering work in the 1960s \cite{BelGliGro:56,Kal:60,BolPon:61}, has become a cornerstone in stochastic control theory, leading to extensive developments in theory, algorithms, and numerical methods. Notably, classical LQ control primarily considers systems governed by forward stochastic differential equations (FSDEs), which describe ``present-determined future'' dynamics but cannot directly address problems constrained by terminal conditions.

In practical applications, such as financial mathematics and risk management, decision-making problems often involve terminal objectives: the present decision must satisfy specified future targets, e.g., achieving a pre-defined portfolio return or controlling terminal risk. These problems are naturally modeled by backward stochastic differential equations (BSDEs). The foundational work of \cite{ParPen:90} established existence and uniqueness results for adapted solutions to BSDEs, providing the theoretical basis for backward stochastic control.

Building on BSDE theory, continuous-time backward stochastic linear-quadratic (BSLQ) control has been systematically developed. For example, \cite{DokZho:99} revealed the dual relationship between continuous-time BSLQ and forward LQ problems under terminal constraints; \cite{LimZho:01} presented a complete solution framework for continuous-time BSLQ, including the derivation of the associated stochastic Riccati equation and feedback optimal control; \cite{LiSunXio:19} extended these results to mean-field BSLQ systems, rigorously proving uniqueness and providing closed-loop solutions; \cite{HuaWanZha:20} incorporated partial observation and information constraints, enriching practical applicability. Recent surveys \cite{SunWenXio:22} summarize these developments, including indefinite-weight cases, and outline future directions in stochastic control and financial engineering. Despite these advances, discrete-time BSLQ control remains relatively underexplored.

Discrete-time frameworks are important in practice because financial transactions and control actions are often implemented at daily or weekly intervals, where continuous-time models provide only approximations by \cite{RasSte:05,KuLeeZhu:12}. To illustrate the practical relevance, consider a simple discrete-time portfolio allocation problem where an investor seeks to optimize the terminal wealth of a portfolio over $N$ periods under stochastic returns and variance constraints. This setup naturally leads to a discrete-time backward stochastic LQ control formulation, highlighting both the need for a backward model and the impact of terminal constraints.

Previous works have focused on numerical approximation and theoretical analysis for discrete-time backward stochastic LQ equations (BS$\Delta$Es) \cite{BouTou:04,Tur:15,GobTur:16,JiaHu:19,CheKawShiYam:23,HanLi:25}. However, explicit discrete-time BSLQ control theory, especially for nonhomogeneous systems with cross terms, is scarce. Notable progress includes \cite{ZhaLiuXuZha:25}, which addressed discrete-time BSLQ under homogeneous assumptions and partial information.

This paper investigates the discrete-time BSLQ optimal control problem with nonhomogeneous terms and cross-term cost functions. Compared with existing studies:

\begin{enumerate}
\item 
Different from  \cite{ZhaLiuXuZha:25} that focuses on homogeneous systems under partial information, this paper further considers nonhomogeneous state equations and cost function cross terms.
\item Unlike \cite{SunWuXio:23}, which addresses continuous-time BSLQ, we focus on the discrete-time setting and provide practical solution frameworks.
\item Unlike \cite{WuTanMen:25}, which studies discrete-time forward LQ systems, our approach directly addresses backward discrete-time dynamics using BS$\Delta$Es.
\end{enumerate}

The main contributions of this paper are as follows:

\begin{enumerate}
\item We extend discrete-time BSLQ models to include nonhomogeneous terms, providing a more realistic description of dynamic systems.
\item We develop a novel equivalent transformation and system decoupling approach to handle asymmetry in the Riccati equations caused by cross terms, yielding explicit Riccati equations, feedback laws, and value function expressions.
\item We demonstrate that the discrete-time framework is consistent with real-world decision-making processes, ensuring implementable solutions.
\end{enumerate}

The remainder of the paper is organized as follows. Section 2 presents the discrete-time BSLQ problem formulation and basic assumptions. Section 3 investigates the existence and uniqueness of optimal controls. Section 4 derives the optimal control and its explicit representation. Section 5 provides a numerical example illustrating the main results. Section 6 concludes the paper and discusses future research directions.
    
    Next, we systematically introduce the required notations.
    $\mathbb{R}$: The set of all real numbers; $\mathbb{R}^{n}$: The $n$-dimensional real vector space; $\mathbb{R}^{n \times m}$: The space of all $n \times m$ real matrices; $\mathbb{S}^{n}$: The set of all $n \times n$ real symmetric matrices; $\mathbb{S}_{+}^{n}$: The set of all positive semi-definite matrices in $\mathbb{S}^{n}$; $A^{\top}$: The transpose of matrix $A$; $A^{-1}$: The inverse of matrix $A$ (if $A$ is invertible); $I_{n}$: The $n \times n$ identity matrix. It can be abbreviated as $I$ when the dimension is clear; $A > 0$ ($A \geq 0$): $A$ is a (semi)positive definite matrix, where $A \in \mathbb{S}^{n}$; $A \gg 0$: $A$ is uniformly positive definite, i.e., there exists a constant $\delta > 0$ such that $A \geq \delta I_{n}$.

     In the control problem studied in this paper, the terminal time of the process is denoted by $N$. We define the time horizon as:
     \[
        \begin{aligned}
            \overline{\mathcal{T}} &= \{0, 1, \cdots, N\}, \\
            \mathcal{T} &= \{0, 1, \cdots, N-1\}.
        \end{aligned}
    \]
    On the matrix space $\mathbb{R}^{n \times m}$, a Frobenius inner product structure can be defined. Its expression is
    \[\langle M, N \rangle = \operatorname{tr}(M^{\top} N), \quad \forall M, N \in \mathbb{R}^{n \times m}.\]
    The norm induced by this inner product is given by the formula
    \[\| M \| = \sqrt{\langle M, M \rangle} = \sqrt{\operatorname{tr}(M^{\top} M)}.\]
    We now introduce some spaces used in this paper. Let $(\Omega, \mathcal{F}, \mathbb{F}, \mathbb{P})$ be a complete probability space endowed with a filtration $\mathbb{F}=\{\mathcal{F}_k\}_{k=0}^{N}$. We assume the filtration is $\mathbb{P}$-complete and generated by a martingale difference sequence $\{\omega_k\}_{k\in\mathcal{T}}$, satisfying the following properties:
    \[
        \begin{aligned}
            \mathcal{F}_k &= \sigma\{\omega_0, \omega_1, \dots, \omega_k\}, \\
            \mathbb{E}[\omega_{k+1} \mid \mathcal{F}_k] &= 0, \quad \mathbb{E}[(\omega_{k+1})^2 \mid \mathcal{F}_k] = 1,
        \end{aligned}
    \]
    with the convention that $\mathcal{F}_{-1} = \{\emptyset, \Omega\}$.
    Let $\mathcal{X} \subset \mathbb{R}^{n \times m}$.
    Denote by $L_{\mathcal{F}_{N-1}}^2(\Omega;\mathcal{X})$ the space of all $\mathcal{X}$-valued, $\mathcal{F}_{N-1}$-measurable random variables $h$ such that
    \[
        \|h\|_{L_{\mathcal{F}_{N-1}}^2(\Omega;\mathcal{X})} := \left( \mathbb{E}\left[ \|h\|_{\mathcal{X}}^2 \right] \right)^{1/2} < \infty.
    \]
    Let $L_{\mathbb{F}}^2(\mathcal{T};\mathcal{X})$ be the space of all $\mathcal{X}$-valued stochastic processes $h = \{h_k\}_{k \in \mathcal{T}}$ such that $h_k$ is $\mathcal{F}_{k-1}$-measurable for each $k \in \mathcal{T}$, and
    \[
        \|h\|_{L_{\mathbb{F}}^2(\mathcal{T};\mathcal{X})} := \left( \mathbb{E} \left[ \sum_{k=0}^{N-1} \|h_k\|_{\mathcal{X}}^2 \right] \right)^{1/2} < \infty.
    \]

      \section{Problem Statement}
    Consider the following discrete-time linear backward stochastic difference equation (BS$\Delta$E)
with control input over the finite time horizon $\mathcal{T}$:
\begin{equation} \label{eq:2.1}
    \begin{cases}
        y_k = A_k\mathbb{E}_{k-1}\bigl[ y_{k+1} \bigr] + B_ku_k \\
        \quad\quad + C_k\mathbb{E}_{k-1}\bigl[ y_{k+1}\omega_k \bigr] + q_k, \\
        y_N = \xi, \quad k \in \mathcal{T}
    \end{cases}
    \end{equation}
In the above, $y_k(\cdot)$ is called the system state process taking values in the Euclidean space $\mathbb{R}^n$,
and $u_k(\cdot)$ is called the control input process taking values in $\mathbb{R}^m$.
The notation $\mathbb{E}_{k-1}[\cdot]$ refers to the conditional expectation
with respect to the information set $\mathcal{F}_{k-1}$ at time $k-1$.
For each $k\in\mathcal{T}$, $A_k$, $B_k$, $C_k$ are deterministic constant matrices of proper dimensions,
$q_k$ is an $\mathbb{F}$-adapted process of proper dimension,
and $\xi\in L_{\mathcal{F}_{N-1}}^2(\Omega;\mathbb{R}^n)$ is the terminal condition.

    We now define the following Hilbert space:
    \[
    \begin{aligned}		
    &\mathcal{U} = \mathcal{U}(\mathcal{T}) \triangleq\bigg\{u= \left(u_{0}, u_{1}, \ldots, u_{N-1}\right) \bigg| u_{k} \text{ is }  \\
    &\mathcal{F}_{k-1}\text{-measurable, } u_{k} \in \mathbb{R}^{m}, \mathbb{E}\left[\sum_{k=0}^{N-1}\left|u_{k}\right|^{2}\right] < \infty\bigg\}.
    \end{aligned}
    \]
    Here, any $u  \in \mathcal{U}$ is designated as an admissible control process. The solution $y$ of state equation \eqref{eq:2.1} corresponding to such a control is called an admissible state process, and the pair $(y, u)$ is termed an admissible pair.

    In the state equation \eqref{eq:2.1}, for the admissible control process $u$ and the given terminal state $\xi$, we define the quadratic cost functional as follows:
    \begin{equation}\label{eq:2.2}
        \begin{aligned}
    	    & \mathcal{J}(N,\xi;u)\\
            & =\frac{1}{2}\mathbb{E}\Bigg\{\left\langle G_{0}y_{0},y_{0}\right\rangle \\
            &+ \sum_{k=0}^{N-1} \bigg[\left\langle \begin{pmatrix} Q_{k} & S_{k}^{\top} \\ S_{k} & R_{k} \end{pmatrix}\begin{pmatrix} \mathbb{E}_{k-1}[y_{k+1}] \\ u_k \end{pmatrix}, \begin{pmatrix} \mathbb{E}_{k-1}[y_{k+1}] \\
            u_k \end{pmatrix}\right\rangle\\
            &+ 2\left\langle \begin{pmatrix} \eta_{k} \\
            \rho_{k} \end{pmatrix}, \begin{pmatrix} \mathbb{E}_{k-1}[y_{k+1}]
            \\ u_k \end{pmatrix}\right\rangle\bigg] \Bigg\}.
        \end{aligned}
    \end{equation}
    Here, $G_0$ is a constant symmetric matrix,
and for each $k\in\mathcal{T}$, $Q_k,S_k,R_k$ are deterministic constant matrices.
Moreover, $\eta_k$ and $\rho_k$ are $\mathcal{F}_{k-1}$-measurable processes.

    \begin{rmk}
        Unlike the standard formulation of backward stochastic difference equations with an explicit martingale integrand, the present system incorporates the martingale term through the conditional expectation
\[
\mathbb{E}_{k-1}\left[y_{k+1}\omega_k\right],
\]
which is standard in discrete-time stochastic control and equivalent to introducing an implicit $Z_k$-process.
    \end{rmk}
	\begin{pro}$\big(DTBSLQ\big)$
		For a given terminal state $\xi \in L_{\mathcal{F}_{N-1}}^{2}(\Omega;\mathbb{R}^{n})$, find an admissible control $u^{*} \in \mathcal{U}$ such that:
		\begin{equation}\label{eq:2.3}
			V(N,\xi) \triangleq \mathcal{J}(N,\xi; u^{*}) = \inf_{u \in \mathcal{U}} \mathcal{J}(N,\xi; u).
		\end{equation}
	\end{pro}
	If there exists $u^{*} \in \mathcal{U}$ satisfying the above equality, then $u^{*}$ is referred to as the open-loop optimal control of the problem (DTBSLQ); the corresponding state process $y^{*}(\cdot) = \{y_{k}^{(\xi,u^{*})}\}_{k=0}^{N}$ is termed the optimal state process; the pair $(y^{*}, u^{*})$ is called the optimal pair; and $V(N,\xi)$ is referred to as the value function of the problem (DTBSLQ).

When the nonhomogeneous terms satisfy $q_k = \eta_k = \rho_k = 0$ for each $k\in\mathcal{T}$,
the original problem reduces to a homogeneous backward stochastic difference equation problem,
denoted as \textup{(DTBSLQ)$^0$}.
The simplified system can be expressed as
\begin{equation}\label{eq:2.4}
	\begin{cases}
	y_k = A_k\mathbb{E}_{k-1}[y_{k+1}] + B_k v_k
	      + C_k\mathbb{E}_{k-1}[y_{k+1}\omega_{k}], \\
	y_N = \xi,\quad k\in\mathcal{T},
	\end{cases}
\end{equation}
and the corresponding cost functional can be formulated as
\begin{equation}\label{eq:2.5}
\begin{aligned}
&\mathcal{J}^{0}(N,\xi;v)
\\&= \frac{1}{2}\mathbb{E}\Bigg\{
 \langle G_{0}y_{0},y_{0}\rangle
\\
&+ \sum_{k=0}^{N-1}\left\langle
        \begin{pmatrix}Q_{k} & S_{k}^{\top} \\ S_{k} & R_{k}\end{pmatrix}
        \begin{pmatrix}\mathbb{E}_{k-1}[y_{k+1}] \\ v_k\end{pmatrix},
        \begin{pmatrix}\mathbb{E}_{k-1}[y_{k+1}] \\ v_k\end{pmatrix}
    \right\rangle \Bigg\}.
\end{aligned}
\end{equation}

    We impose the following conditions on the coefficients of the state equation \eqref{eq:2.1} and the weighting matrices of the cost functional \eqref{eq:2.2}, leading to the fundamental assumptions below.

\begin{ass}\label{ass:2.1}
    For each $k\in\mathcal{T}$, the system coefficients satisfy
\[
A_k, C_k\in\mathbb{R}^{n\times n},\quad B_k\in\mathbb{R}^{n\times m}.
\]
Moreover, the nonhomogeneous term
\[
q(\cdot)\in L_{\mathbb{F}}^2(\mathcal{T};\mathbb{R}^n).
\]
\end{ass}
    \begin{ass}\label{ass:2.2}
   The weighting matrices satisfy, for each $k\in\mathcal{T}$,
\[
G_0\in\mathbb{S}^n,\quad Q_k\in\mathbb{S}^n,\quad R_k\in\mathbb{S}^m,\quad S_k\in\mathbb{R}^{m\times n},
\]
and the nonhomogeneous terms
\[
\eta(\cdot)\in L_{\mathbb{F}}^2(\mathcal{T};\mathbb{R}^n),\quad
\rho(\cdot)\in L_{\mathbb{F}}^2(\mathcal{T};\mathbb{R}^m).
\]
    \end{ass}
    To ensure the well-posedness of state equation \eqref{eq:2.1}, we introduce the following lemma, which is a direct consequence of BS$\Delta$Es; see \cite{NiuMenLiTan:26}.
    \begin{lem}\label{lem:2.3}
		Suppose Assumption \ref{ass:2.1} holds. Then for any terminal condition $\xi \in L_{\mathcal{F}_{N-1}}^{2}(\Omega;\mathbb{R}^{n})$, and any admissible control $u = \{u_{k}\}_{k \in \mathcal{T}} \in \mathcal{U}$, the state equation \eqref{eq:2.1} admits a unique adapted solution $y(\cdot) \equiv y^{(\xi,u)}(\cdot) = \left\{ y_k^{(\xi,u)} \right\}_{k=0}^{N}\in L_{\mathbb{F}}^2(\mathcal{T};\mathbb{R}^{n})$.
		Furthermore, there exists a constant $L > 0$ such that the following estimate holds:
			\begin{equation}\label{eq:2.6}
                \mathbb{E}
                \left[\sum_{k=0}^{N}|y_{k}|^{2}\right] \leq L\,\mathbb{E}\left[|\xi|^{2} + \sum_{k=0}^{N-1}\left(|u_k|^{2} + |q_k|^{2}\right)\right].
                \end{equation}
    \end{lem}

\section{Representation of the Cost Functional and Optimal Control}

In this section, we represent the cost functional in a Hilbert space framework and prepare for the characterization of the optimal control. Throughout this section, we keep the notation introduced in Section 2.

Let the terminal state be $\xi\in L_{\mathcal F_{N-1}}^{2}(\Omega;\mathbb R^n)$, and let the control process be $u=\{u_k\}_{k\in\mathcal T}\in\mathcal U$. Let
\[
\{y_k^{(\xi,u)}\}_{k=0}^{N}
\]
be the unique adapted solution to the state equation \eqref{eq:2.1}.

By the linearity of the state equation with respect to the data and the uniqueness of the adapted solution, the solution $y^{(\xi,u)}$ admits the decomposition
\begin{equation}\label{eq:3.decomp}
y_k^{(\xi,u)}=y_k^\xi+y_k^u+y_k^0,\qquad k=0,1,\dots,N,
\end{equation}
where:
\begin{itemize}
    \item $y^\xi=\{y_k^\xi\}_{k=0}^N$ is the unique adapted solution corresponding to terminal state $\xi$, zero control, and zero inhomogeneous term;
    \item $y^u=\{y_k^u\}_{k=0}^N$ is the unique adapted solution corresponding to zero terminal state, control $u$, and zero inhomogeneous term;
    \item $y^0=\{y_k^0\}_{k=0}^N$ is the unique adapted solution corresponding to zero terminal state, zero control, and the given inhomogeneous term $q(\cdot)$.
\end{itemize}

To formulate the problem in Hilbert spaces, we introduce
\[
\mathscr X_{-1}:=L_{\mathcal F_{-1}}^2(\Omega;\mathbb R^n),\quad
\mathscr X_{N-1}:=L_{\mathcal F_{N-1}}^2(\Omega;\mathbb R^n),\]
\[
\mathscr Z:= L_{\mathbb{F}}^2(\mathcal{T};\mathbb{R}^n),
\]
\[
    \begin{aligned}
\mathscr Y
:=&
\left\{
\phi=\{\phi_k\}_{k\in\mathcal T}
\ \middle|\right.\\&\left.\
\phi_k\in L_{\mathcal F_k}^2(\Omega;\mathbb R^n),\
\mathbb E\sum_{k=0}^{N-1}|\phi_k|^2<\infty
\right\},
\end{aligned}
\]

Define the conditional expectation operator
\[
\Phi:\mathscr Y\to\mathscr Z,
\qquad
(\Phi\phi)_k:=\mathbb E_{k-1}[\phi_k],\qquad k\in\mathcal T.
\]

For the terminal-state component, define
\[
\mathcal I_0:\mathscr X_{N-1}\to \mathscr X_{-1},
\qquad
\mathcal I_0\xi:=y_0^\xi,
\]
\[
\mathcal I_1:\mathscr X_{N-1}\to \mathscr Y,
\qquad
(\mathcal I_1\xi)_k:=y_{k+1}^\xi,\qquad k\in\mathcal T.
\]

For the control-driven component, define
\[
\mathcal L_0:\mathcal U\to \mathscr X_{-1},
\qquad
\mathcal L_0u:=y_0^u,
\]
\[
\mathcal L_1:\mathcal U\to \mathscr Y,
\qquad
(\mathcal L_1u)_k:=y_{k+1}^u,\qquad k\in\mathcal T.
\]

For the inhomogeneous component, let
\[
\mathcal Q_0:=y_0^0\in \mathscr X_{-1},
\qquad
(\mathcal Q_1)_k:=y_{k+1}^0,\qquad k\in\mathcal T,
\]
so that $\mathcal Q_1\in\mathscr Y$.

We also define the pointwise multiplication operators
\[
Q:\mathscr Z\to\mathscr Z,\qquad (Q\phi)_k:=Q_k\phi_k,
\]
\[
S:\mathscr Z\to\mathcal U,\qquad (S\phi)_k:=S_k\phi_k,
\]
\[
R:\mathcal U\to\mathcal U,\qquad (Ru)_k:=R_ku_k,\qquad k\in\mathcal T.
\]
All the above operators are bounded linear operators, which follows from Lemma \ref{lem:2.3}, finite horizon, and bounded weighting matrices.
For any bounded linear operator $\mathcal O$ between Hilbert spaces, let $\mathcal O^*$ denote its adjoint operator.

On $\mathscr X_{-1}$ and $\mathscr X_{N-1}$, we define
\[
[[X,Y]]:=\mathbb E\langle X,Y\rangle.
\]
On $\mathscr Y$, $\mathscr Z$, and $\mathcal U$, we define
\[
[[\phi,\psi]]:=\mathbb E\Big[\sum_{k=0}^{N-1}\langle \phi_k,\psi_k\rangle\Big].
\]
For simplicity, the same notation $[[\cdot,\cdot]]$ is used whenever no confusion arises.

Then
\[
y_0^{(\xi,u)}=\mathcal I_0\xi+\mathcal L_0u+\mathcal Q_0,
\]
and
\[
y_{k+1}^{(\xi,u)}
=(\mathcal I_1\xi)_k+(\mathcal L_1u)_k+(\mathcal Q_1)_k,\qquad k\in\mathcal T.
\]
Hence,
\[
\mathbb E_{k-1}[y_{k+1}^{(\xi,u)}]
=
\big(\Phi(\mathcal I_1\xi+\mathcal L_1u+\mathcal Q_1)\big)_k,
\qquad k\in\mathcal T.
\]

Therefore, the cost functional \eqref{eq:2.2} can be rewritten as
\begin{equation}\label{eq:3.1}
\begin{aligned}
&\mathcal J(N,\xi;u)
\\=&\frac12 [[G_0(\mathcal I_0\xi+\mathcal L_0u+\mathcal Q_0),\mathcal I_0\xi+\mathcal L_0u+\mathcal Q_0]]
\\
&+\frac12 [[Q\Phi(\mathcal I_1\xi+\mathcal L_1u+\mathcal Q_1),\Phi(\mathcal I_1\xi+\mathcal L_1u+\mathcal Q_1)]]
\\
&+[[S\Phi(\mathcal I_1\xi+\mathcal L_1u+\mathcal Q_1),u]]
+\frac12 [[Ru,u]]
\\
&+[[\eta,\Phi(\mathcal I_1\xi+\mathcal L_1u+\mathcal Q_1)]]
+[[\rho,u]]
\\
=&\frac12 [[\mathcal Au,u]]
+[[\mathcal B\xi,u]]
+\frac12 [[\mathcal C\xi,\xi]]
+[[\mathbf a,u]]
+[[\mathbf b,\xi]]
+\mathbf c,
\end{aligned}
\end{equation}
where
\begin{equation}\label{eq:3.2}
\begin{aligned}
\mathcal A
&=\mathcal L_0^*G_0\mathcal L_0
+\mathcal L_1^*\Phi^*Q\Phi\mathcal L_1
+\mathcal L_1^*\Phi^*S^\top
+S\Phi\mathcal L_1
+R,
\\
\mathcal B
&=\mathcal L_0^*G_0\mathcal I_0
+\mathcal L_1^*\Phi^*Q\Phi\mathcal I_1
+S\Phi\mathcal I_1,
\\
\mathcal C
&=\mathcal I_0^*G_0\mathcal I_0
+\mathcal I_1^*\Phi^*Q\Phi\mathcal I_1,
\\
\mathbf a
&=\mathcal L_0^*G_0\mathcal Q_0
+\mathcal L_1^*\Phi^*Q\Phi\mathcal Q_1
+\mathcal L_1^*\Phi^*\eta
+S\Phi\mathcal Q_1
+\rho,
\\
\mathbf b
&=\mathcal I_0^*G_0\mathcal Q_0
+\mathcal I_1^*\Phi^*Q\Phi\mathcal Q_1
+\mathcal I_1^*\Phi^*\eta,
\\
\mathbf c
&=\frac12 [[G_0\mathcal Q_0,\mathcal Q_0]]
+\frac12 [[Q\Phi\mathcal Q_1,\Phi\mathcal Q_1]]
+[[\eta,\Phi\mathcal Q_1]].
\end{aligned}
\end{equation}

In particular, if $q=\eta=\rho=0$, then
\[
\mathcal Q_0=0,\qquad \mathcal Q_1=0,
\]
and hence
\[
\mathbf a=0,\qquad \mathbf b=0,\qquad \mathbf c=0.
\]
Therefore,
\[
\mathcal J^0(N,\xi;u)
=\frac12 [[\mathcal Au,u]]
+[[\mathcal B\xi,u]]
+\frac12 [[\mathcal C\xi,\xi]].
\]
    \subsection{Necessary and Sufficient Conditions for Optimal Control}
    \begin{lem}\label{lem:3.2}
            Suppose Assumptions \ref{ass:2.1} and \ref{ass:2.2} hold. Then, for any given terminal state $\xi \in L_{\mathcal{F}_{N-1}}^{2}(\Omega;\mathbb{R}^{n})$, the quadratic functional \eqref{eq:2.2}  is well-defined and continuous in $u$ over $\mathcal{U}$.
    \end{lem}
        \begin{proof}
        \textbf{(well-defined)} Since the operators $\mathcal{A}$, $\mathcal{B}$, $\mathcal{C}$ and the random variables $\mathbf{a}$, $\mathbf{b}$, $\mathbf{c}$ are linear and bounded, applying the Cauchy-Schwarz inequality yields
        \[
            \begin{aligned}
                |
                &\mathcal{J}(N,\xi,u)| \\\leq& \frac{1}{2} |[[\mathcal{A} u, u]]|+ |[[\mathcal{B} \xi, u]]| + \frac{1}{2} |[[\mathcal{C} \xi, \xi ]]| + |[[\mathbf{a}, u ]]| \\&+ |[[\mathbf{b}, \xi]]| + |\mathbf{c}| \\
            \leq& \frac{1}{2} \|\mathcal{A}\| \|u\|^{2} + \|\mathcal{B}\| \|\xi\| \|u\|+ \frac{1}{2} \|\mathcal{C}\| \|\xi\|^{2} + \|\mathbf{a}\| \|u\|\\& + \|\mathbf{b}\|\|\xi\| + |\mathbf{c}|
            \end{aligned}
        \]
        Therefore, $\mathcal{J}(N,\xi,u)$ is well-defined.

        \textbf{(continuous)}
        Consider two controls $u$ and $\hat{u}\in \mathcal{U}$, and compute the difference:
        \[
            \begin{aligned}
                &
                |\mathcal{J}(N,\xi,u) - \mathcal{J}(N,\xi,\hat{u})| \\=& \left| \frac{1}{2} [[\mathcal{A} u, u]]- \frac{1}{2}[[\mathcal{A} \hat{u}, \hat{u}]] +[[\mathcal{B} \xi, u- \hat{u} ]] + [[\mathbf{a}, u - \hat{u}]]  \right|,
            \end{aligned}
        \]
        where
        \[
            \begin{aligned}
                    &|[[ \mathcal{A} u, u]]  - [[\mathcal{A} \hat{u}, \hat{u}]]|
                    \\=& |[[\mathcal{A} (u - \hat{u}), u]]  + [[\mathcal{A} \hat{u}, u - \hat{u}]] | \\
                    \leq& \|\mathcal{A}\| \|u - \hat{u}\| \|u\| + \|\mathcal{A}\| \|\hat{u}\| \|u - \hat{u}\|\\
                    =& \|\mathcal{A}\| \|u - \hat{u}\| (\|u\| + \|\hat{u}\|),
            \end{aligned}
        \]
        \[[[\mathcal{B} \xi, u - \hat{u}]] \leq \|\mathcal{B}\| \|\xi\| \|u - \hat{u}\|,\]
        \[[[\mathbf{a}, u - \hat{u} ]] \leq \|\mathbf{a}\|\|u - \hat{u}\|.\]
        Combining the estimates above, there exists a constant $\bar{L} > 0$ such that:
        \[
            \begin{aligned}
        &|\mathcal{J}(N,\xi,u) - \mathcal{J}(N,\xi,\hat{u})| \\\leq& \bar{L} \|u - \hat{u}\| \left( \|u\|+ \|\hat{u}\| + \|\xi\| + 1 \right).
        \end{aligned}
        \]
        Therefore, when $\hat{u} \rightarrow u$ in $\mathcal{U}$, the right-hand side tends to 0, hence
        \[\mathcal{J}(N,\xi,\hat{u}) \rightarrow \mathcal{J}(N,\xi,u).\] Continuity is thus proved.
    \end{proof}

    \begin{lem}\label{lem:3.3}
        Let Assumptions \ref{ass:2.1} and \ref{ass:2.2} hold. Then, for any admissible control $u\in\mathcal{U}$, the cost functional $\mathcal{J}(N,\xi;u)$ is Fr\'echet differentiable at $u$, and its derivative is given by
        \begin{equation}\label{eq:3.5}
            [[\mathcal{J}^{\prime}(N,\xi;u),v]]=[[\mathcal{A}u+\mathcal{B}\xi+\mathbf{a},v]],
        \end{equation}
        where $\mathcal{A},\mathcal{B},\mathbf{a}$ are defined by \eqref{eq:3.2}.
    \end{lem}

    \begin{proof}
        For all $\delta \in (0,1)$, applying \eqref{eq:3.1}, we have
        \begin{equation}\label{eq:3.6}
            \mathcal{J}(N,\xi;u+\delta v)-\mathcal{J}(N,\xi;u)=\delta\mathscr{L}^{(u,v)}+\delta^{2}\mathcal{J}^{0}(N,0;v).
        \end{equation}
        Where
        \[
            \mathscr{L}^{(u,v)} = [[\mathcal{A}u+\mathcal{B}\xi+\mathbf{a},v]],
        \]
        \[
            \mathcal{J}^{0}(N,0;v)=\frac{1}{2}[[\mathcal{A}v,v]].
        \]
        Especially, setting $\delta=1$, we can obtain
        \[
            \mathcal{J}(N,\xi;u+v)-\mathcal{J}(N,\xi;u)=\mathscr{L}^{(u,v)}+\mathcal{J}^{0}(N,0;v).
        \]
        Due to the boundedness of $\mathcal{A}$, there exists a constant $M>0$ such that
        \[
            \mathcal{J}^{0}(N,0;v)=\frac{1}{2}[[\mathcal{A}v,v]]\leq\frac{M}{2}\|v\|^{2}.
        \]
        Therefore
        \[
            \begin{aligned}
            \frac{\left|\mathcal{J}(N,\xi;u+v)-\mathcal{J}(N,\xi;u)-\mathscr{L}^{(u,v)}\right|}{\|v\|} \leq \frac{M}{2} \|v\| \to 0 \\\text{as} \quad \|v\| \to 0.
            \end{aligned}
            \]
        This satisfies the definition of Fréchet differentiability, and its derivative is given by:
        \[
            [[\mathcal{J}'(N,\xi;u), v]] = \mathscr{L}^{(u,v)} = [[\mathcal{A}u+\mathcal{B}\xi+\mathbf{a},v]],
        \]
        which implies that \eqref{eq:3.5} holds.
    \end{proof}

    \begin{cor}\label{cor:3.3}
        Let Assumptions \ref{ass:2.1} -- \ref{ass:2.2} hold. Then, for any $\xi \in L_{\mathcal{F}_{N-1}}^{2}(\Omega;\mathbb{R}^{n})$, $\delta \in \mathbb{R}$, and $u, v \in \mathcal{U}$, the following relationship between $\mathcal{J}(N, \xi; u + \delta v)$ and $\mathcal{J}(N, \xi; u)$ holds:
        \begin{equation}\label{eq:3.4}
        \begin{aligned}
            &\mathcal{J}(N,\xi;u+\delta v)\\=&\mathcal{J}(N,\xi;u)+\delta^{2}\mathcal{J}^{0}(N,0;v)+\delta[[\mathcal{J}^{\prime}(N,\xi;u),v]],
        \end{aligned}
        \end{equation}
        where $[[\mathcal{J}^{\prime}(N,\xi;u),v]]$ given by \eqref{eq:3.5} is also the G\^ateaux derivative of $\mathcal{J}(N,\xi;u)$ at $u$ in the direction of $v$.
    \end{cor}

    \begin{proof}
        From Lemma \ref{lem:3.3}, we have obtained \eqref{eq:3.4}. From equation \eqref{eq:3.6}, we obtain:
        \[
            \begin{aligned}
            &\lim_{\delta \to 0} \frac{\mathcal{J}(N,\xi;u+\delta v)-\mathcal{J}(N,\xi;u)}{\delta}\\=&\mathscr{L}^{(u,v)}=[[\mathcal{J}^{\prime}(N,\xi;u),v]],
            \end{aligned}
            \]
        which coincides with the Fréchet derivative. Therefore, the expression in equation \eqref{eq:3.5} is also the Gâteaux derivative of $\mathcal{J}$ at $u$ in the direction $v$.
    \end{proof}

    \begin{lem}\label{lem:3.4}
        Let Assumptions \ref{ass:2.1} -- \ref{ass:2.2} hold. Suppose there exists a constant $r > 0$ such that $\mathcal{J}^{0}(N,0;u)$ satisfies
        \begin{equation}\label{eq:3.7}
            \mathcal{J}^{0}(N,0;u) \geq r \mathbb{E}\left[\sum_{k=0}^{N-1} |u_k|^2\right] = r \|u\|_{\mathcal{U}}^2 \quad \forall u \in \mathcal{U}
        \end{equation}
        then the cost functional $\mathcal{J}(N,\xi;u)$ is strictly convex and coercive on the control space $\mathcal{U}$.
    \end{lem}
    \begin{proof}
       By Corollary \ref{cor:3.3}, we set $v \in \mathcal{U}$, $\delta = 1$, $u = (0, \dots, 0)$, and replace \(v\) with \(u\) (since \(v\) is arbitrary, this replacement is valid). Thus, we obtain:
        \[
            \mathcal{J}(N, \xi; v) = \mathcal{J}(N, \xi; 0) + \mathcal{J}^{0}(N, 0; v) + [[ \mathcal{J}'(N, \xi; 0),v ]].
        \]
        Due to the arbitrariness of $v$, the cost functional can also be written as
        \begin{equation}\label{eq:3.8}
            \mathcal{J}(N, \xi; u) = \mathcal{J}(N, \xi; 0) + \mathcal{J}^{0}(N, 0; u) + [[\mathcal{J}'(N, \xi; 0), u]].
        \end{equation}
        The expression $\mathcal{J}(N,\xi;0) +[[\mathcal{J}'(N,\xi;0), u]]$ is linear in $u$, and hence convex. Since $\mathcal{J}^0(N,0;u)$ is a quadratic functional and satisfies \eqref{eq:3.7}, it follows that $\mathcal{J}^0(N,0;u)$ is strictly convex. Therefore, from \eqref{eq:3.8} we conclude that $\mathcal{J}(N,\xi;u)$ is strictly convex.

        Combining \eqref{eq:3.7} and  \eqref{eq:3.8}, it follows that
         \[
            \begin{aligned}
            &\displaystyle\lim_{\|v\|_{\mathcal{U}}\rightarrow \infty}\mathcal{J}(N,\xi;v)\\=&\displaystyle\lim_{\|v\|_{\mathcal{U}}\rightarrow \infty}\{\mathcal{J}(N,\xi;0)+\mathcal{J}^0(N,0;v)+[[\mathcal{J}'(N,\xi;0), v]]\}\\=&\infty,
            \end{aligned}
            \]
        which implies that  the cost functional   $\mathcal{J}(N,\xi;u)$ is  coercive with respect to  $u$ over $ \mathcal{U}$.
        The proof is complete.
    \end{proof}

    \begin{thm}\label{thm:3.5}
    Let Assumptions \ref{ass:2.1} -- \ref{ass:2.2} hold. Then,  for a given terminal state $\xi \in L^{2}_{\mathcal{F}_{N-1}}(\Omega;\mathbb{R}^{n})$, a control $u^{*} \in \mathcal{U}$ is optimal if and only if
    \begin{enumerate}
        \item[(i)] $\mathcal{J}^{0}(N,0;v)\geq 0,\quad\forall v \in \mathcal{U}$;
        \item[(ii)] $[[\mathcal{J}^{\prime}(N,\xi;u^{*}), v ]]= 0,\quad\forall v \in \mathcal{U}$.
    \end{enumerate}
\end{thm}
    \begin{proof}
        \textbf{(Necessity)}
        Assume $u^{*}$ is optimal. Let $v \in \mathcal{U}$ be an arbitrary admissible control and $\lambda\in\mathbb{R}$. Consider the controls $u^{*} + \lambda v$ and $u^{*} - \lambda v$. We obtain
        \[
            \mathcal{J}(N,\xi;u^{*}+\lambda v) - \mathcal{J}(N,\xi;u^{*})\geq 0.
        \]
        \[
            \mathcal{J}(N,\xi;u^{*}-\lambda v) - \mathcal{J}(N,\xi;u^{*})\geq 0.
        \]
        According to the definition of the Gâteaux derivative, the derivative in the direction $v$ is
        \[[[\mathcal{J}^{\prime}(N, \xi; u^{*}), v]] = \lim_{\lambda \to 0^{+}} \frac{\mathcal{J}(N, \xi; u^{*} + \lambda v) - \mathcal{J}(N, \xi; u^{*})}{\lambda}.\]
        From this, we obtain
        \begin{equation}\label{eq:3.9}
            \begin{aligned}
                [[\mathcal{J}^{\prime}(N, \xi; u^{*}), v]] &\geq 0,  \\
                [[\mathcal{J}^{\prime}(N, \xi; u^{*}), -v]] &\geq 0.
            \end{aligned}
        \end{equation}
        Since $[[\mathcal{J}^{\prime}(N, \xi; u^{*}), v]]$ is linear with respect to the direction $v$, we have
        \begin{equation}\label{eq:3.10}
            [[\mathcal{J}^{\prime}(N, \xi; u^{*}), -v]] = -[[\mathcal{J}^{\prime}(N, \xi; u^{*}), v]] \geq 0.
        \end{equation}
        Combining  \eqref{eq:3.9} and \eqref{eq:3.10},
        \[0 \leq [[\mathcal{J}^{\prime}(N, \xi; u^{*}), v]] \leq 0,\]
        thus,
        \[[[\mathcal{J}^{\prime}(N,\xi;u^{*}), v ]]= 0, \quad \forall v \in \mathcal{U}.\]

        We next prove (i).
        According to \eqref{eq:3.4}, we have
        \[
        \begin{aligned}
        &\mathcal{J}(N,\xi;u^{*}+\lambda v) \\=& \mathcal{J}(N,\xi;u^{*}) + \lambda^{2}\mathcal{J}^{0}(N,0;v)+ \lambda[[\mathcal{J}^{\prime}(N,\xi;u^{*}), v]],
        \end{aligned}
        \]
        By the optimality of $u^{*}$, we have
\[
        \begin{aligned}
        &\mathcal{J}(N,\xi;u^{*}+\lambda v)-\mathcal{J}(N,\xi;u^{*})\\=&\lambda^{2}\mathcal{J}^{0}(N,0;v)+ \lambda[[\mathcal{J}^{\prime}(N,\xi;u^{*}), v]]\\\geq &0.
        \end{aligned}
        \]
        Substituting (ii) into the expansion and dividing by $\lambda^2>0$ yields
        \[
        \mathcal{J}^0(N,0;v) \ge 0.
        \]

        \textbf{(Sufficiency)}
        Assume (i) and (ii) hold. By Corollary \ref{cor:3.3}, for any admissible control $v\in\mathcal{U}$,
\[
            \begin{aligned}
&\mathcal{J}(N,\xi;v)
\\=&\mathcal{J}(N,\xi;u^*)+\mathcal{J}^0(N,0;v-u^*)
+[[\mathcal{J}'(N,\xi;u^*),v-u^*]].
            \end{aligned}
            \]
Since
\[
\begin{aligned}
\mathcal{J}^0(N,0;v-u^*)\ge0,\quad [[\mathcal{J}'(N,\xi;u^*),v-u^*]]=0,
\end{aligned}
\]
 we get
\[
\mathcal{J}(N,\xi;v)\ge \mathcal{J}(N,\xi;u^*).
\]
        Hence, $u^{*}$ is optimal.
\end{proof}

    \begin{rmk}
        According to Theorem \ref{thm:3.5},
        \[\mathcal{J}^{0}(N,0;u) = \frac{1}{2} [[\mathcal{A}u, u]]   \geq 0, \quad \forall u \in \mathcal{U}\]
        (equivalent to the positivity condition $\mathcal{A} \geq 0$) is a necessary condition for the existence of an optimal control.

        If the operator $\mathcal{A}$ is uniformly positive, i.e., there exists a constant $r > 0$ such that
        \[\mathcal{J}^{0}(N, 0; u) = \frac{1}{2} [[\mathcal{A} u, u]] \geq r \, \mathbb{E}\left[ \sum_{k=0}^{N-1} |u_k|^2 \right] = r \|u\|_{\mathcal{U}}^2,\]
        Lemmas \ref{lem:3.2}--\ref{lem:3.4} confirm $\mathcal{J}(N, \xi; u)$ is well-defined, continuous, strictly convex, and coercive in $\mathcal{U}$ for $u$.
By Proposition 2.12 \cite{EkeTem:76}, any $\xi \in L_{\mathcal{F}_{N-1}}^{2}(\Omega;\mathbb{R}^{n})$ gives a unique optimal control $u^*$. According to Theorem \ref{thm:3.5}, we have
        \[[[ \mathcal{J}'(N, \xi; u^{*}), v ]] = [[\mathcal{A}u^{*}+\mathcal{B}\xi+\mathbf{a},v]]= 0, \quad \forall v \in \mathcal{U}.\]
        Since $\mathcal{A}$ is uniformly positive, the optimal control uniquely exists and is given by:
        \[u^{*} = -\mathcal{A}^{-1}(\mathcal{B}\xi + \mathbf{a}).\]
        \end{rmk}
    In this paper, we make the following assumption:
    \begin{ass}\label{ass:3.1}
        For \( k \in \mathcal{T} \), the following conditions hold:
        \[
            \begin{cases}
                G_0\ge 0, R_k \gg 0, \\
                Q_k - S_k^\top R_k^{-1} S_k \ge 0, & k \in \mathcal{T}.
            \end{cases}
        \]
    \end{ass}

    	\section{Stochastic Hamiltonian System and Stochastic Riccati equation }
	This section derives the explicit solution to the discrete-time BSLQ optimal control problem. We first establish the maximum principle for the simplified system and deduce the coupled discrete-time stochastic Hamiltonian system (a set of forward and backward stochastic difference equations (FBS$\Delta$Es)). After performing equivalent transformation on the original problem, we apply a decoupling strategy to derive the discrete Riccati equation, present the optimal control’s explicit state feedback law, and finally obtain the analytical expression of the value function.

    \subsection{Stochastic Hamiltonian System}
	To further analyze the optimal control problem, we introduce the adjoint equation and the Hamiltonian function.
    The adjoint equation associated with the state equation \eqref{eq:2.1} is determined by the following FS$\Delta$E:
    \begin{equation}\label{eq:4.1}
        \left\{\begin{array}{ll}
            x_{k+1} = A_k^\top x_k + Q_k \mathbb{E}_{k-1}[y_{k+1}] + S_k^\top u_k + \eta_k \\\quad\quad \quad+ C_k^\top x_k \omega_k,
            \\x_0 = G_0 y_0, \quad\quad k \in \mathcal{T},
        \end{array}
    \right.
    \end{equation}
    The Hamiltonian function $\mathcal{H}$ is defined as:
    \[\mathcal{H} :\mathcal{T} \times \mathbb{R}^n \times \mathbb{R}^{m} \times \mathbb{R}^n \times \mathbb{R}^n\to \mathbb{R},\]
    and given by:
    \[
    \begin{aligned}
        &\mathcal{H}(k, x, u, y,z)\\ =& \ \left\langle  A_k y+ B_k u + C_k z+ q_k , x \right\rangle+ \frac{1}{2}[ \left\langle Q_k  y, y \right\rangle \\&+ 2 \left\langle S_k y,u \right\rangle + \left\langle R_k u, u \right\rangle + 2 \left\langle \eta_k, y \right\rangle + 2 \left\langle \rho_k, u \right\rangle ].
    \end{aligned}
    \]
    Using the Hamiltonian, the adjoint equation \eqref{eq:4.1} can be rewritten as:
    \begin{equation}\label{eq:4.2}
        \left\{\begin{array}{ll}
             x_{k+1} = \mathcal{H}_y(k) + \mathcal{H}_z(k) \omega_k,
            \\ x_0 = G_0 y_0, \quad k \in \mathcal{T},
        \end{array}
    \right.
    \end{equation}
    where the partial derivatives of the Hamiltonian function are given by:
    \[
    \begin{aligned}
        \mathcal{H}_y(k)&=\frac{\partial \mathcal{H}}{\partial y}(k, x_{k}, u_k, \mathbb{E}_{k-1}[y_{k+1}],\mathbb{E}_{k-1}[y_{k+1} \omega_{k}])\\&=A_k^{\top} x_{k}+Q_{k} \mathbb{E}_{k-1}[y_{k+1}]+S_{k}^{\top} u_k+\eta_{k}
    \end{aligned}
    \]
    \[
    \begin{aligned}
        \mathcal{H}_z(k)&=\frac{\partial \mathcal{H}}{\partial z}(k, x_{k}, u_k, \mathbb{E}_{k-1}[y_{k+1}],\mathbb{E}_{k-1}[y_{k+1} \omega_{k}])\\&=C_k^{\top} x_{k}
    \end{aligned}
    \]
    We then apply the Pontryagin maximum principle: at optimal control, the Hamiltonian’s gradient with respect to control variable $u$ is zero, as expressed by the following stationarity condition.
     \begin{thm}\label{thm:4.1}
     Let Assumptions \ref{ass:2.1}, \ref{ass:2.2} and \ref{ass:3.1} hold, and let $\xi \in L_{\mathcal{F}_{N-1}}^2(\Omega; \mathbb{R}^n)$ be a given terminal state.
The control process $u^* = \{u_k^*\}_{k \in \mathcal{T}}$ is optimal with respect to $\xi$ if and only if
there exists a pair $(x^*, y^*)$ satisfying the Hamiltonian system \eqref{eq:4.3} and the stationarity condition \eqref{eq:4.4}.
        \begin{equation}\label{eq:4.3}
            \left\{\begin{array}{ll}
                 x_{k+1}^* = A_k^{\top} x_k^* + Q_k \mathbb{E}_{k-1}[y_{k+1}^*] + S_k^{\top} u_k^* + \eta_k \\\quad\quad \quad+ C_k^{\top} x_k^* \omega_{k},
                \\y_k^* = A_k \mathbb{E}_{k-1}[y_{k+1}^*] + B_k u_k^* \\\quad\quad+ C_k \mathbb{E}_{k-1}[y_{k+1}^* \omega_{k}] + q_k,
                \\x_0^* = G_0 y_0^*, \quad y_N^* = \xi, \quad k \in \mathcal{T}.
            \end{array}\right.
        \end{equation}
            \begin{equation}\label{eq:4.4}
                B_k^\top x_k^* + S_k \mathbb{E}_{k-1}[y_{k+1}^*] + R_k u^*_k + \rho_k = 0, \quad k \in \mathcal{T}.
            \end{equation}
    \end{thm}
    \begin{proof}
        Let \( u^* \) be an admissible control. Then for any control \( v \in \mathcal{U} \) and any \( \delta\in \mathbb{R} \), we have
        \[\mathcal{J}(N, \xi; u^* + \delta v) - \mathcal{J}(N, \xi; u^*) \ge 0,\]
        Using Lemma \ref{lem:3.3}, we obtain:
    \[
        \begin{aligned}
        &\mathcal{J}(N, \xi; u^* + \delta v) - \mathcal{J}(N, \xi; u^*) \\
        &= \delta\mathbb{E}\Bigg\{  \langle G_0 y_0^*, y_0^{(v)} \rangle + \\&\quad\sum_{k=0}^{N-1} \left[\left\langle \begin{pmatrix}Q_{k} & S_{k}^{\top} \\ S_{k} & R_{k}\end{pmatrix} \begin{pmatrix} \mathbb{E}_{k-1}[y_{k+1}^*] \\ u_k^* \end{pmatrix}, \begin{pmatrix} \mathbb{E}_{k-1}[y_{k+1}^{(v)}] \\ v_k \end{pmatrix} \right\rangle \right. \\
        &\quad + \left.  \left\langle \begin{pmatrix} \eta_k \\ \rho_k \end{pmatrix}, \begin{pmatrix} \mathbb{E}_{k-1}[y_{k+1}^{(v)}] \\ v_k \end{pmatrix} \right\rangle \right]\Bigg\}+ \frac{1}{2} \delta^2 \, \mathbb{E} \Bigg\{ \langle G_0 y_0^{(v)}, y_0^{(v)} \rangle  \\
        &\quad + \sum_{k=0}^{N-1} \left\langle \begin{pmatrix}Q_{k} & S_{k}^{\top} \\ S_{k} & R_{k}\end{pmatrix} \begin{pmatrix} \mathbb{E}_{k-1}[y_{k+1}^{(v)}] \\ v_k \end{pmatrix}, \begin{pmatrix} \mathbb{E}_{k-1}[y_{k+1}^{(v)}] \\ v_k \end{pmatrix} \right\rangle \Bigg\}
         \\
        &= \delta [[\mathcal{J}'(N, \xi; u^*), v]]   + \delta^2 \mathcal{J}^0(N, 0; v).
        \end{aligned}
    \]
        Here,
        \begin{equation}\label{eq:4.5}
        \begin{aligned}
        &[[\mathcal{J}'(N, \xi; u^*), v]] \\=& \mathbb{E} \bigg[ \langle G_0 y_0^*, y_0^{(v)} \rangle\\& + \sum_{k=0}^{N-1} \left( \langle Q_k \mathbb{E}_{k-1}[y_{k+1}^*] + S_k^\top u_k^* + \eta_k, \mathbb{E}_{k-1}[y_{k+1}^{(v)}] \rangle \right. \\
        &+\left. \langle S_k \mathbb{E}_{k-1}[y_{k+1}^*] + R_k u_k^* + \rho_k, v_k \rangle \right) \bigg],
        \end{aligned}
        \end{equation}
        and
\[
        \begin{aligned}
        \quad &\mathcal{J}^0(N, 0; v) \\=& \frac{1}{2} \mathbb{E} \Bigg\{ \langle G_0 y_0^{(v)}, y_0^{(v)} \rangle
\\
        &+  \sum_{k=0}^{N-1} \left( \langle Q_k \mathbb{E}_{k-1}[y_{k+1}^{(v)}] + S_k^\top v_k, \mathbb{E}_{k-1}[y_{k+1}^{(v)}] \rangle \right.
        \\&+ \left. \langle S_k \mathbb{E}_{k-1}[y_{k+1}^{(v)}] + R_k v_k, v_k \rangle \right) \Bigg\}.
        \end{aligned}
        \]
        Summing $\langle x_{k}^{*}, y_{k}^{(v)} \rangle - \langle x_{k+1}^{*}, y_{k+1}^{(v)} \rangle$ from $k=0$ to $N-1$ , we have:
        \begin{equation}\label{eq:4.6}			
        \begin{aligned}
            &\mathbb{E} \left\{ \langle G_{0} y_{0}^{*}, y_{0}^{(v)} \rangle \right\} \\=& \mathbb{E} \left\{ \langle x_{0}^{*}, y_{0}^{(v)} \rangle \right\} \\
            =& \mathbb{E} \bigg\{ \sum_{k=0}^{N-1} \left[ \langle x_{k}^{*}, y_{k}^{(v)} \rangle - \langle x_{k+1}^{*}, y_{k+1}^{(v)} \rangle \right] \bigg\}
            \\
            =& \mathbb{E} \bigg\{ \sum_{k=0}^{N-1} \left[ \left\langle x_{k}^{*}, A_k \mathbb{E}_{k-1}[y_{k+1}^{(v)}] + B_k v_k \right.\right. \\
        &\left.\left.+ C_k \mathbb{E}_{k-1}[y_{k+1}^{(v)} \omega_{k}] \right\rangle - \langle x_{k+1}^{*}, y_{k+1}^{(v)} \rangle \right] \bigg\} \\
            =& \mathbb{E} \bigg\{ \sum_{k=0}^{N-1} \left[ \left\langle A_k^{\top} x_k^* - x_{k+1}^* + C_k^\top x_k^* \omega_k, \mathbb{E}_{k-1}[y_{k+1}^{(v)}] \right\rangle\right.\\&\left. + \left\langle B_k^{\top} x_k^*,\ v_k \right\rangle\right] \bigg\}
        \end{aligned}
        \end{equation}
         Substituting equation \eqref{eq:4.6} into equation \eqref{eq:4.5}, we have
        \[
        \begin{aligned}
        &[[ \mathcal{J}'(N,\xi;u^*),v ]]
            \\=& \mathbb{E}\Bigg\{ \sum_{k=0}^{N-1}  \Bigl[\langle A_k^\top x_k^* - x_{k+1}^* + Q_k \mathbb{E}_{k-1}[y_{k+1}^*] + S_k^\top u_k^*\\& + \eta_k + C_k^\top x_k^* \omega_k, \mathbb{E}_{k-1}[y_{k+1}^{(v)}] \rangle +  \langle B_k^\top x_k^* \\
            &+ S_k \mathbb{E}_{k-1}[y_{k+1}^*] + R_k u_k^* + \rho_k, v_k \rangle\Bigl] \Bigg\}
        \end{aligned}
        \]
        Applying the adjoint equation \eqref{eq:4.1}, we obtain
        \begin{equation}\label{eq:4.61}
    \begin{aligned}
    &[[\mathcal{J}'(N,\xi;u^*), v]] \\=& \mathbb{E}\Bigg\{ \sum_{k=0}^{N-1}  \left\langle B_k^{\top} x_{k}^{*} + S_{k} \mathbb{E}_{k-1}[y_{k+1}^{*}] + R_{k} u_k^{*}\right.\\&\left. + \rho_{k}, v_k \right\rangle \Bigg\}
    \end{aligned}
    \end{equation}
        \textbf{(Necessity)}
        Let $(y^*,u^*)$ be an optimal pair of Problem (DTBSLQ), and let $x^*$ be the solution to the adjoint equation \eqref{eq:4.1} corresponding to this optimal pair $(y^*,u^*)$. According to Theorem \ref{thm:3.5}, when $u^*$ is the optimal control, $[[ \mathcal{J}'(N,\xi;u^*),v ]]=0$, which implies
        \[
        \begin{aligned}
        &\mathbb{E}\Bigg\{ \sum_{k=0}^{N-1}\left\langle B_k^{\top}x_k^*+S_k\mathbb{E}_{k-1}[y_{k+1}^*]+R_ku_k^*+\rho_k,\ v_k\right\rangle \Bigg\}\\&=0.
        \end{aligned}
        \]
        Due to the arbitrariness of $v$, the inner product factor of the integrand must be zero for each $k\in\mathcal{T}$. Therefore, we obtain
        \[B_k^{\top}x_k^*+S_k\mathbb{E}_{k-1}[y_{k+1}^*]+R_ku_k^*+\rho_k=0.\]
        The necessity is thus proved.

        \textbf{(Sufficiency)}
    Let $(y^*, u^*)$ be an admissible pair of the Problem (DTBSLQ), and let $x^*$ be the solution to the adjoint equation \eqref{eq:4.1} corresponding to $(y^*, u^*)$. Under Assumptions \ref{ass:2.1}, \ref{ass:2.2} and \ref{ass:3.1}, we have
    \[ \mathcal{J}^0(N,0;v) \geq 0, \quad \forall v \in \mathcal{U}. \]
    This confirms condition (i) of Theorem \ref{thm:3.5}. Next, suppose that the pair $(x^*, y^*)$ satisfies the Hamiltonian system \eqref{eq:4.3}, and the stationarity condition \eqref{eq:4.4} holds, i.e.,
    \[B_k^\top x_k^* + S_k \mathbb{E}_{k-1}[y_{k+1}^*] + R_k u_k^* + \rho_k = 0, \quad \forall k \in \mathcal{T}.\]

    Substituting this stationarity condition into the expression of \eqref{eq:4.61}, we immediately obtain that
    \[ [[\mathcal{J}'(N,\xi;u^*), v]] = 0, \quad \forall v \in \mathcal{U}. \]
    This verifies condition (ii) in Theorem \ref{thm:3.5}.

    Therefore, by Theorem \ref{thm:3.5}, the control $u^*$ is the optimal control for the given terminal state $\xi$. The sufficiency is proved.
        \end{proof}
         \begin{rmk}
    The condition
    \[B_k^{\top} x_k^* + S_k \mathbb{E}_{k-1}[y_{k+1}^*] + R_k u_k^* + \rho_k = 0\]
    is called the stationarity condition. Since \( R_k > 0 \) by Assumption \ref{ass:3.1}, \( R_k \) is invertible. Thus, the optimal control admits the feedback form:
    \begin{equation}\label{eq:4.7}
    u_k^* = -R_k^{-1} \left[ B_k^{\top} x_k^* + S_k \mathbb{E}_{k-1}[y_{k+1}^*] + \rho_k \right], \quad k \in {\mathcal{T}}.
    \end{equation}
    \end{rmk}
   Next, we first perform a preliminary simplification on the expression of the value function.

    Applying the state equation \eqref{eq:2.1} and the adjoint equation \eqref{eq:4.1}, compute the sum of the differences $\langle x^{*}_{k}, y^{*}_{k} \rangle - \langle x^{*}_{k+1}, y^{*}_{k+1} \rangle$ from $k=0$ to $k=N-1$.
                \begin{equation}\label{eq:4.8}
			\begin{aligned}
				& \mathbb{E} \left\{ \langle G_{0} y_{0}^{*}, y_{0}^{*} \rangle \right\} \\
				= & \mathbb{E} \left\{ \langle x_{0}^{*}, y_{0}^{*} \rangle \right\} \\
				= & \mathbb{E} \left\{ \sum_{k=0}^{N-1} \left[ \langle x_{k}^{*}, y_{k}^{*} \rangle - \langle x_{k+1}^{*}, y_{k+1}^{*} \rangle \right] + \langle x_{N}^{*}, y_{N}^{*} \rangle \right\} \\
				= & \mathbb{E} \left\{ \sum_{k=0}^{N-1} \left[ \langle x_{k}^{*}, A_k \mathbb{E}_{k-1} [y_{k+1}^{*}] + B_k u_k^*\right.\right.\\&\left.\left. + C_k \mathbb{E}_{k-1} [y_{k+1}^{*} \omega_k] + q_k \rangle - \langle x_{k+1}^{*}, y_{k+1}^{*} \rangle \right] \right.\\&\left.+ \langle x_{N}^{*}, y_{N}^{*} \rangle \right\}
                \\
				= & \mathbb{E} \left\{ \sum_{k=0}^{N-1} \left[ \langle A_k^{\top} x_{k}^{*} -x_{k+1}^{*} + C_k^{\top} x_{k}^{*} \omega_k, \mathbb{E}_{k-1} [y_{k+1}^{*}] \rangle \right.\right.\\&\left.\left.+ \langle B_k^{\top} x_{k}^{*}, u_k^{*} \rangle + \langle x_{k}^{*}, q_k \rangle \right] + \langle x_{N}^{*}, y_{N}^{*} \rangle \right\} \\
				= & \mathbb{E} \left\{ \sum_{k=0}^{N-1} \left[ -\langle Q_{k} \mathbb{E}_{k-1} [y_{k+1}^{*}]+  S_k^{\top} u_k^* +\eta_{k}, \mathbb{E}_{k-1} [y_{k+1}^{*}] \rangle \right.\right.
            \\&\left.\left.+ \langle B_k^{\top} x_{k}^{*}, u_k^{*} \rangle + \langle x_{k}^{*}, q_k \rangle \right] + \langle x_{N}^{*}, y_{N}^{*} \rangle \right\}
			\end{aligned}
		\end{equation}

		Below, using the relation \eqref{eq:4.8} and the stationarity condition \eqref{eq:4.4}, the value function \( V(N,\xi) \) is expressed in a more concise form.
\begin{equation}\label{eq:4.9}
			\begin{aligned}
            	&V(N, \xi) \\=&\mathcal{J}(N,\xi;u^\ast)
				\\
				=&\frac{1}{2}\mathbb{E}\left\langle G_{0}y_{0}^{*},y_{0}^{*}\right\rangle +\frac{1}{2}\mathbb{E}\Bigg\{\sum_{k=0}^{N-1}\Bigl[\left\langle Q_{k}\mathbb{E}_{k-1}[y_{k+1}^{*}],\mathbb{E}_{k-1}[y_{k+1}^{*}]\right\rangle \\&+\left\langle R_{k}u_k^{*},u_k^{*}\right\rangle+2 \left\langle S_{k}\mathbb{E}_{k-1} [y_{k+1}^{*}],u_k^{*}\right\rangle\\&+2\left\langle \eta_{k},\mathbb{E}_{k-1}[y_{k+1}^{*}]\right\rangle +2\left\langle \rho_{k},u_k^{*}\right\rangle\Bigl] \Bigg\} \\
				=& \frac{1}{2} \mathbb{E}\Bigg\{ \sum_{k=0}^{N-1} \Bigl[ -\langle Q_{k} \mathbb{E}_{k-1} [y_{k+1}^{*}]+  S_k^{\top} u_k^* + \eta_k, \mathbb{E}_{k-1} [y_{k+1}^{*}] \rangle \\&+ \langle B_k^{\top} x_{k}^{*}, u_k^{*} \rangle + \langle x_{k}^{*}, q_k \rangle \Bigl]+ \langle x_{N}^{*}, y_{N}^{*} \rangle  \Bigg\} 
                \\
				&+ \frac{1}{2} \mathbb{E}\Bigg\{ \sum_{k=0}^{N-1} \Bigl[ \langle Q_{k} \mathbb{E}_{k-1} [y_{k+1}^{*}], \mathbb{E}_{k-1} [y_{k+1}^{*}] \rangle \\&+ \langle R_{k} u_k^{*}, u_k^{*} \rangle +2 \left\langle S_{k}\mathbb{E}_{k-1} [y_{k+1}^{*}],u_k^{*}\right\rangle\\
				&+ 2 \langle \eta_{k}, \mathbb{E}_{k-1} [y_{k+1}^{*}] \rangle + 2 \langle \rho_{k}, u_k^{*} \rangle \Bigl] \Bigg\}
				\\=& \frac{1}{2} \mathbb{E} \Bigg\{ \sum_{k=0}^{N-1} \Bigl[ \left\langle B_k^{\top}x_{k}^{*}+S_k \mathbb{E}_{k-1}[y_{k+1}^*]+ R_{k} u_k^{*} \right.\\&\left.+ \rho_{k}, u_k^{*} \right\rangle + \left\langle x_{k}^{*}, q_k \right\rangle + \left\langle \eta_{k},\mathbb{E}_{k-1}[y_{k+1}^{*}]\right\rangle
\\
				&+\left\langle \rho_{k}, u_k^{*}\right\rangle\Bigl]+ \left\langle x_{N}^{*}, y_{N}^{*} \right\rangle \Bigg\}
\\
				=&\frac{1}{2}\mathbb{E}\Bigg\{\sum_{k=0}^{N-1}\Bigl[\left\langle x_{k}^{*}, q_k\right\rangle+\left\langle \eta_{k}, \mathbb{E}_{k-1}[y_{k+1}^{*}]\right\rangle\\&+\left\langle \rho_{k}, u_k^{*}\right\rangle\Bigl]+\left\langle x_{N}^{*}, y_{N}^{*}\right\rangle\Bigg\}.
			\end{aligned}
		\end{equation}

            \subsection{Equivalent Problem of Problem (DTBSLQ)}
	Since the solution to the Riccati equation corresponding to the cost function with cross terms lacks symmetry in the backward stochastic linear quadratic optimal control problem, we derive an equivalent form of the original problem (DTBSLQ) in this subsection.

    First, let $H = \{ H_k \}_{k=0}^N \in \mathbb{S}$ be the unique solution to the following equations:
    \begin{equation}\label{eq:4.10}
    \begin{cases}
    \left( A_k^\top H_k B_k \right) R_k^{-1} \left( B_k^\top H_k A_k \right) + A_k^\top H_k A_k \\- H_{k+1} = 0, \\
    H_0 = G_0, ~~~~k \in \mathcal{T}.
    \end{cases}
    \end{equation}
    Through the following transformations:
    \[
    \begin{aligned}
    \bar{u}_k &= u_k + R_k^{-1} \left( B_k^\top H_k A_k + S_k \right) \mathbb{E}_{k-1}[y_{k+1} \omega_k], \\
    \bar{C}_{k} &= C_k - B_kR_k^{-1} \left( B_k^\top H_k A_k + S_k \right), \\
    \hat{Q}_k &= Q_k - S_k^\top R_k^{-1} S_k \\&\quad + (A_k^\top H_k B_k) R_k^{-1} (B_k^\top H_k A_k), \\
    \hat{S}_k &= -B_k^\top H_k A_k, \\
    \hat{\eta}_k &= \eta_k - \left( A_k^\top H_k B_k + S_k^\top \right) R_k^{-1} \rho_k,
    \end{aligned}
    \]
    the original problem (DTBSLQ) is transformed into the following BSLQ optimal control problem with state equation:
    \begin{equation}\label{eq:4.12}
    \begin{cases}
    y_k = A_k \mathbb{E}_{k-1}[y_{k+1}] + B_k \bar{u}_k + \bar{C}_{k} \mathbb{E}_{k-1}[y_{k+1} \omega_k]\\\quad\quad + q_k, \\
    y_N = \xi,~~~~k \in \mathcal{T},
    \end{cases}
    \end{equation}
    and cost functional
    \begin{equation}\label{eq:4.13}
    \begin{aligned}
    &\mathcal{J}(N, \xi; \bar{u})\\ =& \frac{1}{2} \mathbb{E} \Bigg\{ \langle G_0 y_0, y_0 \rangle \\&+ \sum_{k=0}^{N-1} \Bigg\langle \begin{pmatrix} \hat{Q}_k & \hat{S}_k^\top \\ \hat{S}_k & R_k \end{pmatrix} \begin{pmatrix} \mathbb{E}_{k-1}[y_{k+1}] \\ \bar{u}_k \end{pmatrix}, \begin{pmatrix} \mathbb{E}_{k-1}[y_{k+1}] \\ \bar{u}_k \end{pmatrix} \Bigg\rangle \\
    &+ 2 \Bigg\langle \begin{pmatrix} \hat{\eta}_k \\ \rho_k \end{pmatrix}, \begin{pmatrix} \mathbb{E}_{k-1}[y_{k+1}] \\ \bar{u}_k \end{pmatrix} \Bigg\rangle \Bigg\}.
    \end{aligned}
    \end{equation}
    Let \(y = \{ y_k \}_{k=0}^N\) be the solution to the state equation \eqref{eq:4.12}. By summing \(\langle H_k y_k, y_k \rangle - \langle H_{k+1} y_{k+1}, y_{k+1} \rangle\) over \(k = 0\) to \(k = N-1\), we obtain:
    \[
    \begin{aligned}
     &\mathbb{E} \left\{ \langle G_0 y_0, y_0 \rangle \right\}\\
    =&\mathbb{E} \left\{ \langle H_0 y_0, y_0 \rangle \right\}\\
    =&\mathbb{E} \biggl\{ \sum_{k=0}^{N-1} \Bigl[ \langle H_k y_k, y_k \rangle - \langle H_{k+1} y_{k+1}, y_{k+1} \rangle \Bigl] \\&+ \langle H_N y_N, y_N \rangle \biggl
    \} \\
    =& \mathbb{E} \biggl\{ \sum_{k=0}^{N-1} \Bigl[ \langle H_k ( A_k  \mathbb{E}_{k-1}[y_{k+1}]+ B_k \bar{u}_k\\&\left.+ \hat{C}_k \mathbb{E}_{k-1}[y_{k+1} \omega_k] + q_k ),( A_k  \mathbb{E}_{k-1}[y_{k+1}]\right.\\&\left.+ B_k \bar{u}_k + \hat{C}_k \mathbb{E}_{k-1}[y_{k+1} \omega_k] + q_k) \rangle \right.\\&- \langle H_{k+1} y_{k+1}, y_{k+1} \rangle \Bigl]+ \langle H_N y_N, y_N \rangle \biggl\}
    \\
    =& \mathbb{E} \left\{ \sum_{k=0}^{N-1} \left[\left\langle \begin{pmatrix}
    A_k^\top H_k A_k - H_{k+1} & A_k^\top H_k B_k \\
    B_k^\top H_k A_k & B_k^\top H_k B_k
    \end{pmatrix}\right.\right.\right.\\& \left.\left.\left. \begin{pmatrix}
    \mathbb{E}_{k-1}[y_{k+1}] \\ \bar{u}_k
    \end{pmatrix}, \right. \right.\begin{pmatrix}
    \mathbb{E}_{k-1}[y_{k+1}] \\ \bar{u}_k
    \end{pmatrix} \right\rangle 
    \end{aligned}\]
    
    \begin{equation}\label{eq:4.14}
    \begin{aligned}
    &\left.\left.+ 2\left\langle \begin{pmatrix}
    A_k^\top H_k q_k \\ B_k^\top H_k q_k
    \end{pmatrix}, \begin{pmatrix}
    \mathbb{E}_{k-1}[y_{k+1}] \\ \bar{u}_k
    \end{pmatrix} \right\rangle \right.\right.\\&\left.+ \langle H_k q_k, q_k \rangle \Bigl] + \langle H_N y_N, y_N \rangle \right\}.
    \end{aligned}
    \end{equation}
    We define the following notations:
    \[
    \begin{aligned}
    \bar{Q}_k &= Q_k - S_k^\top R_k^{-1} S_k, \\
    \bar{R}_k &= R_k + B_k^\top H_k B_k, \\
    \bar{\eta}_k &= \eta_k - \left( A_k^\top H_k B_k + S_k^\top \right) R_k^{-1} \rho_k + A_k^\top H_k q_k, \\
    \bar{\rho}_k &= \rho_k +B_k^\top H_k q_k.
    \end{aligned}
    \]
    It follows easily from Assumption \ref{ass:3.1} that $\bar{Q}_k \ge 0$ and $\bar{R}_k \gg 0$.

    Now, by substituting \eqref{eq:4.14} into \eqref{eq:4.13} and using \eqref{eq:4.10} , we obtain:
            \begin{equation}\label{eq:4.16}
        \begin{aligned}
            &\quad \mathcal{J}(N, \xi; \bar{u})\\
            =& \frac{1}{2} \mathbb{E} \Bigg\{ \sum_{k=0}^{N-1} \bigg[
            \left\langle \begin{pmatrix}
            A_k^\top H_k A_k - H_{k+1} & A_k^\top H_k B_k \\
            B_k^\top H_k A_k & B_k^\top H_k B_k
            \end{pmatrix}\right.\\&\left. \begin{pmatrix}
            \mathbb{E}_{k-1} [y_{k+1}] \\ \bar{u}_k
            \end{pmatrix}, \right.\begin{pmatrix}
            \mathbb{E}_{k-1} [y_{k+1}] \\ \bar{u}_k
            \end{pmatrix} \rangle \\
            &+ 2\left\langle \begin{pmatrix}
            A_k^\top H_k q_k \\ B_k^\top H_k q_k
            \end{pmatrix}, \begin{pmatrix}
            \mathbb{E}_{k-1} [y_{k+1}] \\ \bar{u}_k
            \end{pmatrix}   \right\rangle\\&+ \langle H_k q_k, q_k \rangle
        \\&+\Bigg\langle \begin{pmatrix} \hat{Q}_k & \hat{S}_k^\top \\ \hat{S}_k & R_k \end{pmatrix} \begin{pmatrix} \mathbb{E}_{k-1}[y_{k+1}] \\ \bar{u}_k \end{pmatrix}, \begin{pmatrix} \mathbb{E}_{k-1}[y_{k+1}] \\ \bar{u}_k \end{pmatrix} \Bigg\rangle \\
            &+ 2 \Bigg\langle
            \begin{pmatrix}
            \hat{\eta}_k \\
            \rho_k
            \end{pmatrix},
            \begin{pmatrix}
            \mathbb{E}_{k-1}[y_{k+1}] \\
            \bar{u}_k
            \end{pmatrix} \Bigg\rangle\bigg]
            \\&+ \langle H_N y_N, y_N \rangle \Bigg\}\\
            =& \frac{1}{2} \mathbb{E} \Bigg\{ \sum_{k=0}^{N-1}\bigg[ \Bigg\langle
            \begin{pmatrix}
            \bar{Q}_k & 0 \\
            0 & \bar{R}_k
            \end{pmatrix}
            \begin{pmatrix}
            \mathbb{E}_{k-1}[y_{k+1}] \\
            \bar{u}_k
            \end{pmatrix},
            \begin{pmatrix}
            \mathbb{E}_{k-1}[y_{k+1}] \\
            \bar{u}_k
            \end{pmatrix} \Bigg\rangle\\&+ 2 \Bigg\langle
            \begin{pmatrix}
            \bar{\eta}_k \\
            \bar{\rho}_k
            \end{pmatrix},
            \begin{pmatrix}
            \mathbb{E}_{k-1}[y_{k+1}] \\
            \bar{u}_k
            \end{pmatrix} \Bigg\rangle + \langle H_k q_k, q_k \rangle\bigg]\\
            &+ \langle H_N y_N, y_N \rangle\Bigg\}
       \end{aligned}
    \end{equation}

    Therefore, for a given terminal state $\xi$, minimizing $\mathcal{J}(N, \xi; u)$ \eqref{eq:2.2} subject to \eqref{eq:2.1} is equivalent to minimizing $\mathcal{J}(N, \xi; \bar{u})$ \eqref{eq:4.16} subject to \eqref{eq:4.12}.From Theorem \ref{thm:4.1},
    the equivalent Hamiltonian system we now consider is given by
    \begin{equation}\label{eq:4.17}
        \begin{cases}
            x_{k+1}^* = A_k^{\top} x_k^* + \bar{Q}_k \mathbb{E}_{k-1}[y_{k+1}^*] + \bar{\eta}_k + \bar{C}_{k}^{\top} x_k^* \omega_k, \\
            y_k^* = A_k \mathbb{E}_{k-1}[y_{k+1}^*] + B_k \bar{u}_k^* \\\quad\quad+ \bar{C}_{k} \mathbb{E}_{k-1}[y_{k+1}^* \omega_k] + q_k, \\
            x_0^* = G_0 y_0^*, \quad y_N^* = \xi, \\
            B_k^\top x_k^* + \bar{R}_k \bar{u}_k^* + \bar{\rho}_k = 0.
        \end{cases}
    \end{equation}

    \subsection{Riccati Equation and Feedback Representation of Optimal Control}
	This subsection focuses on the core contribution of this research: decoupling the fully coupled Hamiltonian system \eqref{eq:4.17} by introducing Riccati difference equation, thereby deriving the feedback form of the optimal control, and then deducing the analytical expression of the value function.Next, we consider the state \eqref{eq:4.12} and the cost function \eqref{eq:4.16}. Furthermore, according to Equation \eqref{eq:4.9}, the concise form of the value function to be discussed next is given by
     \begin{equation}\label{eq:4.18}
		\begin{aligned}
            &V(N, \xi) \\=&\frac{1}{2}\mathbb{E}\Bigg\{\sum_{k=0}^{N-1}\Bigl[\left\langle x_{k}^{*}, q_k\right\rangle+\left\langle \bar{\eta}_{k}, \mathbb{E}_{k-1}[y_{k+1}^{*}]\right\rangle+\left\langle \bar{\rho}_{k}, \bar{u}_k^{*}\right\rangle \\&+ \langle H_k q_k, q_k \rangle \Bigl]+\left\langle x_{N}^{*}, y_{N}^{*}\right\rangle+ \langle H_N y_N, y_N \rangle\Bigg\}
        \end{aligned}
    \end{equation}
    To simplify the notation, for $\Sigma := \{\Sigma_k \mid k \in \mathcal{T}\}\in\mathbb{S}_+^n$, we define the following mappings:
    \begin{equation}\label{eq:4.19}
        \begin{aligned}
            \Theta(\Sigma_{k+1}) &= I + \bar{Q}_k \Sigma_{k+1}, \\
            \mathcal{R}(\Sigma_{k+1}) &= -A_k \Sigma_{k+1} \Theta^{-1}(\Sigma_{k+1}) A_k^{\top}.
        \end{aligned}
    \end{equation}
 \begin{rmk}
        Since $\bar{Q}_k \geq 0$, there exists a matrix $D_k$ such that $\bar{Q}_k = D_k D_k^\top$. Then, by the matrix inversion lemma\cite{TylSoh:05}, we obtain:
        \begin{equation}\label{eq:4.21}
        \begin{aligned}
            &\left(I + D_k D_k^\top \Sigma_{k+1}\right)^{-1} \\=& I - D_k \left(I + D_k^\top \Sigma_{k+1} D_k\right)^{-1} D_k^\top \Sigma_{k+1}.
            \end{aligned}
        \end{equation}
        Since $I + D_k^\top \Sigma_{k+1} D_k > 0$ is invertible, the matrix inversion lemma implies that $\Theta(\Sigma_{k+1})$ is also invertible.
 \end{rmk}

    Before constructing the optimal control for the backward stochastic differential equation, we first introduce the following Riccati equation:
    \begin{equation}\label{eq:4.20}
        \begin{cases}
            \Sigma_k = -\mathcal{R}(\Sigma_{k+1}) + B_k \bar{R}_k^{-1}B_k^{\top}, \quad k \in \mathcal{T}. \\
            \Sigma_N = 0,
        \end{cases}
    \end{equation}
    \begin{thm}\label{thm:4.3}
        Let Assumptions \ref{ass:2.1},\ref{ass:2.2} and \ref{ass:3.1} hold. Then the Riccati equation \eqref{eq:4.20} admits a unique semi-positive definite solution $\Sigma := \{\Sigma_k \mid k \in \mathcal{T} \}$, and furthermore, $\Theta(\Sigma_{k+1})$ is invertible.
    \end{thm}
    \begin{proof} \textbf{(Uniqueness)}
        Assume that $\Sigma^{(1)} = \{\Sigma_k^{(1)} \mid k \in \mathcal{T}\}$ and $\Sigma^{(2)} = \{\Sigma_k^{(2)} \mid k \in \mathcal{T}\}$ are two solutions of the Riccati equation. Let $\Delta \Sigma_k = \Sigma_k^{(1)} - \Sigma_k^{(2)}$, with $\Delta \Sigma_N = 0$.

	    First, from
        \[
        \Delta \Theta(\Sigma_{k+1})= \Theta(\Sigma_{k+1}^{(1)}) - \Theta(\Sigma_{k+1}^{(2)}) = \bar{Q}_k \Delta \Sigma_{k+1},
        \]
        we have
        \[
            \begin{aligned}
                &\Delta \Theta^{-1}(\Sigma_{k+1}) \\= &\Theta^{-1}(\Sigma_{k+1}^{(1)}) - \Theta^{-1}(\Sigma_{k+1}^{(2)}) \\
                =& -\Theta^{-1}(\Sigma_{k+1}^{(1)}) \bar{Q}_k \Delta \Sigma_{k+1} \Theta^{-1}(\Sigma_{k+1}^{(2)}).
            \end{aligned}
        \]

        We further obtain
        \[
            \begin{aligned}
                &\Delta \Sigma_k \\=&-\mathcal{R}(\Sigma_{k+1}^{(1)}) + \mathcal{R}(\Sigma_{k+1}^{(2)})\\
                    =& A_k \left[ \Delta\Sigma_{k+1} \Theta^{-1}(\Sigma_{k+1}^{(1)})\right.\\&\left. - \Sigma_{k+1}^{(2)} \Theta^{-1}(\Sigma_{k+1}^{(1)}) \bar{Q}_k \Delta \Sigma_{k+1} \Theta^{-1}(\Sigma_{k+1}^{(2)}) \right] A_k^{\top}
            \end{aligned}
        \]
        Then, we proceed by backward induction starting from $\Delta \Sigma_{N} = 0$. By backward recursion, we obtain:
        \[
            \Delta \Sigma_{k} = 0 \quad \text{for all } k \in \mathcal{T}
        \]
        This completes the proof of the uniqueness of the solution.

        \textbf{(Existence)}
        We use mathematical induction to prove that the solution to the Riccati equation is a semi-positive definite matrix.

        From \eqref{eq:4.20}, we have:
        \[
            \begin{aligned}
                \Sigma_{N-1} &= -A_{N-1} \Sigma_N \Theta^{-1}(\Sigma_N) A_{N-1}^\top+B_{N-1} \bar{R}_{N-1}^{-1} B_{N-1}^\top \\
                &= B_{N-1} \bar{R}_{N-1}^{-1} B_{N-1}^\top \geq 0.
            \end{aligned}
        \]
        Suppose that at time $k+1$, the solution satisfies $\Sigma_{k+1} \geq 0$.

        We now prove that $\Sigma_k \geq 0$ at time $k$.
        From equations \eqref{eq:4.20} and \eqref{eq:4.21}, we obtain:
        \[
            \begin{aligned}
                \Sigma_k &= A_k \Sigma_{k+1} \left(I + \bar{Q}_k \Sigma_{k+1}\right)^{-1} A_k^\top \\&+ B_k \bar{R}_k^{-1} B_k^\top \\
                &= \left[A_k^\top - D_k \left(I + D_k^\top \Sigma_{k+1} D_k\right)^{-1} D_k^\top \Sigma_{k+1}\right.\\&\left.\quad A_k^\top\right]^\top \Sigma_{k+1} \left[A_k^\top - D_k \left(I + D_k^\top \Sigma_{k+1} D_k\right)^{-1} \right.\\&\quad\left.D_k^\top \Sigma_{k+1} A_k^\top\right]+ A_k \Sigma_{k+1} D_k \left(I + D_k^\top \Sigma_{k+1} D_k\right)^{-1}  \\
                &\quad D_k^\top \Sigma_{k+1} A_k^\top- A_k \Sigma_{k+1} D_k \left(I + D_k^\top \Sigma_{k+1} D_k\right)^{-1}  \\
                &\quad D_k^\top \Sigma_{k+1} D_k \left(I + D_k^\top \Sigma_{k+1} D_k\right)^{-1} D_k^\top \Sigma_{k+1} A_k^\top \\&\quad+ B_k \bar{R}_k^{-1} B_k^\top \\
                &= \left[A_k^\top - D_k \left(I + D_k^\top \Sigma_{k+1} D_k\right)^{-1} D_k^\top \Sigma_{k+1} A_k^\top\right]^\top \\&\quad\Sigma_{k+1} \left[A_k^\top - D_k \left(I + D_k^\top \Sigma_{k+1} D_k\right)^{-1} D_k^\top \Sigma_{k+1}\right.\\&\quad\left. A_k^\top\right] + A_k \Sigma_{k+1} D_k \left(I + D_k^\top \Sigma_{k+1} D_k\right)^{-1}\\&\quad \left(I + D_k^\top \Sigma_{k+1} D_k\right)^{-1} D_k^\top \Sigma_{k+1} A_k^\top + B_k \bar{R}_k^{-1} B_k^\top.
            \end{aligned}
        \]

        Given that $\Sigma_{k+1} \geq 0$ and $\bar{R}_k \gg 0$, we thus have $\Sigma_k \geq 0$.By backward recursion starting from $\Sigma_N = 0$, there exists a solution $\Sigma := \{\Sigma_k \mid k \in \mathcal{T}\}$ satisfying equation \eqref{eq:4.20}.Thus, the existence is proved.
    \end{proof}

    We assume that the optimal state \(y^*\) and the adjoint process \(x^* = \{x_k^*\}_{k=0}^N\) exhibit a linear structural relationship:
    \begin{equation}\label{eq:4.22}
	    y_k^* = -\Sigma_k x_k^* + \varphi_k,
    \end{equation}

	where \(\Sigma := \{\Sigma_k \mid k \in \mathcal{T}\}\) takes values in the space of symmetric positive semidefinite matrices \(\mathbb{S}_+^n\), with the terminal condition \(\Sigma_N = 0\); and \(\varphi = \{\varphi_k \mid k \in \mathcal{T}\}\) is another stochastic process satisfying the following backward stochastic difference equation:
    \begin{equation}\label{eq:4.23}
	    \begin{cases}
	        \varphi_k = f_{k+1}^{(1)} + f_{k+1}^{(2)} \mathbb{E}_{k-1} [\varphi_{k+1}]+f_{k+1}^{(3)}\mathbb{E}_{k-1} [ \varphi_{k+1}\omega_k], \\
	        \varphi_N = \xi,
	    \end{cases}
    \end{equation}

	where \(f_{k+1}^{(1)}\), \(f_{k+1}^{(2)}\),\(f_{k+1}^{(3)}\)are coefficients to be determined.

    For the subsequent calculations, we first substitute relation \eqref{eq:4.22} into the adjoint process given by \eqref{eq:4.17} to express \( x_{k+1} \) in terms of \( x_k \). Then, by the definition of the notation in \eqref{eq:4.19} and the established invertibility of $\Theta(\Sigma_{k+1})$ in Theorem \ref{thm:4.3}, we obtain:
    \begin{equation}\label{eq:4.24}
    \begin{aligned}
        x_{k+1}^* = &\Theta^{-1} \left( \Sigma_{k+1} \right) \left( A_k^{\top} x_k^* + \bar{Q}_k  \mathbb{E}_{k-1}[\varphi_{k+1}]  + \bar{\eta}_{k} \right.\\&\left.+ \bar{C}_{k}^{\top} x_k^* \omega_k \right)
        \end{aligned}
    \end{equation}

    Using the state equation and the adjoint equation in\eqref{eq:4.17}, the notation \eqref{eq:4.19}, and the relation \eqref{eq:4.24}, we obtain:
    \begin{equation}\label{eq:4.25}
        \begin{aligned}
            &-\Sigma_k x_k^* + \varphi_k \\=& y_k^* \\
            =& A_k \mathbb{E}_{k-1} [y_{k+1}^*] + B_k \bar{u}_k^* + \bar{C}_{k} \mathbb{E}_{k-1} [y_{k+1}^* \omega_k] \\&+ q_k \\
            =& A_k \mathbb{E}_{k-1} \left[ -\Sigma_{k+1} \Theta^{-1} \left( \Sigma_{k+1} \right) \left( A_k^{\top}x_k^* \right.\right.\\&\left.\left.+ \bar{Q}_k\mathbb{E}_{k-1}[ \varphi_{k+1}] + \bar{\eta}_{k} + \bar{C}_{k}^{\top} x_k^* \omega_k \right) + \varphi_{k+1} \right] \\&+ B_k \bar{u}_k^* +\bar{C}_{k}\mathbb{E}_{k-1} [ \varphi_{k+1}\omega_k]+ q_k \\
            =& \mathcal{R}(\Sigma_{k+1}) x_k^* + B_k \bar{u}_k^* + \Pi(\Sigma_{k+1}, \varphi_{k+1})
        \end{aligned}
    \end{equation}

     where
    \begin{equation}\label{eq:4.26}
        \begin{aligned}
            &\Pi(\Sigma_{k+1}, \varphi_{k+1}) \\=& \left[A_k-A_k  \Sigma_{k+1} \Theta^{-1} \left( \Sigma_{k+1} \right) \bar{Q}_k\right]  \mathbb{E}_{k-1}[\varphi_{k+1}]\\&+\bar{C}_{k}\mathbb{E}_{k-1} [ \varphi_{k+1}\omega_k]\\&- A_k \Sigma_{k+1} \Theta^{-1} \left( \Sigma_{k+1} \right) \mathbb{E}_{k-1}[\bar{\eta}_{k}]+ q_k
        \end{aligned}
    \end{equation}

    Next, Substitute the stationarity condition in \eqref{eq:4.17} into \eqref{eq:4.22} to obtain
        \begin{equation}\label{eq:4.27}
            \begin{aligned}
                &-\Sigma_{k}x_{k}^*+\varphi_{k}\\ =& \mathcal{R}(\Sigma_{k+1})x_{k}^* + B_k\bar{u}_k^* + \Pi(\Sigma_{k+1},\varphi_{k+1})\\
                =& \left[ \mathcal{R}(\Sigma_{k+1}) - B_k\bar{R}_k^{-1} B_k^{\top}\right]x_{k}^*- B_k\bar{R}_k^{-1}\bar{\rho}_k\\&+ \Pi(\Sigma_{k+1},\varphi_{k+1})
            \end{aligned}
        \end{equation}

    A comparison of both sides of \eqref{eq:4.27} reveals that $\Sigma$ is given by the solution to the following Riccati equation:
    \[
		\left\{\begin{array}{lll}
			\begin{aligned}
				\Sigma_k
				=&A_k \Sigma_{k+1} (I + \bar{Q}_k \Sigma_{k+1})^{-1} A_k^{\top}\\&+ B_k \bar{R}_k^{-1}B_k^{\top} , \quad   k \in \mathcal{T}
				\\
				\Sigma_N =& 0.
			\end{aligned}
		\end{array}
		\right.
        \]

    For a given \( k \in \mathcal{T} \), the process \( \varphi \) is governed by the following backward stochastic differential equation (BSDE):
    \begin{equation}\label{eq:4.29}
		\left\{\begin{array}{lll}
			\begin{aligned}
				\varphi_{k}
				=&\left[A_k-A_k  \Sigma_{k+1} \Theta^{-1} \left( \Sigma_{k+1} \right) \bar{Q}_k\right]  \mathbb{E}_{k-1}[\varphi_{k+1}]\\&+\bar{C}_{k}\mathbb{E}_{k-1} [ \varphi_{k+1}\omega_k] - A_k \Sigma_{k+1} \Theta^{-1} \left( \Sigma_{k+1} \right) \\&\mathbb{E}_{k-1}[\bar{\eta}_{k}]-B_k\bar{R}_{k}^{-1}\bar{\rho}_{k} + q_k, \quad   k \in \mathcal{T}
				\\
				\varphi_{N} =& \xi.
			\end{aligned}
		\end{array}
		\right.
	\end{equation}

	It follows that the terms $ f_{k+1}^{(1)} $, $ f_{k+1}^{(2)} $, $ f_{k+1}^{(3)} $in \eqref{eq:4.23} must satisfy
	\[
	\left\{
	\begin{aligned}
	f_{k+1}^{(1)} = &- A_k \Sigma_{k+1} \Theta^{-1} \left( \Sigma_{k+1} \right) \mathbb{E}_{k-1}[\bar{\eta}_{k}]-B_k\bar{R}_{k}^{-1}\bar{\rho}_{k}\\& + q_k \\
	f_{k+1}^{(2)} = &A_k-A_k  \Sigma_{k+1} \Theta^{-1} \left( \Sigma_{k+1} \right) \bar{Q}_k
     \\
	f_{k+1}^{(3)} = &\bar{C}_{k}
	\end{aligned}
	\right.
	\]

	\begin{thm}\label{thm:4.4}
		  Let Assumptions \ref{ass:2.1},\ref{ass:2.2} and \ref{ass:3.1} hold. Let $\{\Sigma_k \mid k \in \mathcal{T}\} \subset \mathbb{S}_{+}^{n}$ be the solution of the Riccati equation \eqref{eq:4.20}, and let $\{\varphi_k \mid k \in \mathcal{T}\} \subset L_{\mathcal{F}}^2(\mathcal{T}; \mathbb{R}^n)$ be the solution of the backward stochastic differential equation \eqref{eq:4.29}. Then the value function of the problem is given by
		\begin{equation}\label{eq:4.30}
			\begin{aligned}
            	&V(N, \xi)
   				\\=& \frac{1}{2} \mathbb{E} \left\{\sum_{k=0}^{N-1}\Bigl[\langle -\bar{Q}_{k} \Sigma_{k+1} \Theta^{-1}(\Sigma_{k+1})\bar{Q}_{k} \varphi_{k+1}  + \bar{Q}_{k} \varphi_{k+1} \right.\\&+ \bar{\eta}_{k}, \varphi_{k+1} \left\rangle  \right.- \langle \bar{\eta}_{k}, \Sigma_{k+1} \Theta^{-1}(\Sigma_{k+1})  \bar{\eta}_{k}\rangle- \langle \bar{\rho}_{k},\bar{R}_{k}^{-1} \bar{\rho}_{k} \rangle\Bigr]\\
   				&+ \langle H_N \xi, \xi \rangle+ \langle (I + G_0 \Sigma_0)^{-1} G_0 \varphi_0, \varphi_0 \rangle\Bigg\}.
			\end{aligned}
		\end{equation}
	\end{thm}
	\begin{proof}
    	Noting that
\begin{equation}\label{eq:4.31}
\left[\Sigma_{k+1} \Theta^{-1}(\Sigma_{k+1})\right]^\top
= \Sigma_{k+1} \Theta^{-1}(\Sigma_{k+1}),
\end{equation}
and from
\[
x_0 = G_0 y_0 = -G_0 \Sigma_0 x_0 + G_0 \varphi_0,
\]
we derive
\[
x_0 = (I + G_0 \Sigma_0)^{-1} G_0 \varphi_0.
\]
Here, $I + G_0 \Sigma_0$ is invertible, which can be proven following Remark 4.

		We next proceed according to the relations \eqref{eq:4.22},  \eqref{eq:4.24}.
		By summing the expression
		$\langle x_{k+1}^{*}, y_{k+1} \rangle - \langle x_{k}^{*}, y_{k} \rangle$
		from $k = 0$ to $k = N$, we compute the expectation
		\begin{equation}\label{eq:4.32}
			\begin{aligned}
				&\mathbb{E} \Bigl[ \langle x_N^*, y_N^* \rangle \Bigr] \\=& \mathbb{E} \Bigl[ \langle x_N^*, \xi \rangle \Bigl] \\
				=& \mathbb{E} \Bigg\{ \sum_{k=0}^{N-1} \Bigl[\langle x_{k+1}^*, \varphi_{k+1} \rangle - \langle x_k^*, \varphi_k \rangle   \Bigl]+ \langle x_0^*, \varphi_0 \rangle\Bigg\} \\
				=& \mathbb{E} \Bigg\{ \sum_{k=0}^{N-1} \Bigl[\langle A_k^{\top} x_k^* + \bar{Q}_k \mathbb{E}_{k-1} [y_{k+1}] + \bar{\eta}_k\\& + \bar{C}_{k}^{\top} x_k^* \omega_k, \varphi_{k+1} \rangle - \langle x_k^*, \varphi_k \rangle \Bigl]+ \langle x_0^*, \varphi_0 \rangle\Bigg\}\\
				=&\mathbb{E}\Bigg\{ \sum_{k=0}^{N-1}  \Bigl[\langle  A_k (I - \Sigma_{k+1} \Theta^{-1}(\Sigma_{k+1}) \bar{Q}_k)  \mathbb{E}_{k-1} [\varphi_{k+1}] \\&+\bar{C}_{k}\mathbb{E}_{k-1} [ \varphi_{k+1}\omega_k]- \varphi_k, x_k^*\rangle\\&+ \langle -\bar{Q}_k \Sigma_{k+1} \Theta^{-1}(\Sigma_{k+1})( \bar{Q}_k \varphi_{k+1}-
				\bar{\eta}_k) + \bar{Q}_k \varphi_{k+1}\\&+ \bar{\eta}_k, \varphi_{k+1} \rangle \Bigl]+ \langle (I + G_0 \Sigma_0)^{-1} G_0 \varphi_0, \varphi_0 \rangle \Bigg\}
			\end{aligned}
		\end{equation}
		Substitute equation \eqref{eq:4.32} into equation \eqref{eq:4.18}, and by utilizing the relations \eqref{eq:4.22}, \eqref{eq:4.24}, and the stationarity condition in \eqref{eq:4.17} , we obtain the following value function:
        \[
                \begin{aligned}
            	&\quad V(N, \xi)\\ =&\frac{1}{2}\mathbb{E}\Bigg\{\sum_{k=0}^{N-1}\bigg[\left\langle x_{k}^{*}, q_k\right\rangle+\left\langle \bar{\eta}_{k}, \mathbb{E}_{k-1}[y_{k+1}^{*}]\right\rangle+\left\langle \bar{\rho}_{k}, \bar{u}_k^{*}\right\rangle \\&+ \langle H_k q_k, q_k \rangle \bigg]+\left\langle x_{N}^{*}, y_{N}^{*}\right\rangle+ \langle H_N y_N, y_N \rangle\Bigg\}\\
            	 =& \frac{1}{2} \mathbb{E}\Bigg\{ \sum_{k=0}^{N-1} \bigg[\left\langle   A_k( I - \Sigma_{k+1}\Theta^{-1}(\Sigma_{k+1})\bar{Q}_{k} ) \mathbb{E}_{k-1}[\varphi_{k+1}]\right.\\&\left.+ \bar{C}_{k}\mathbb{E}_{k-1} [ \varphi_{k+1}\omega_k]- \varphi_{k}, x_{k}^{*} \right\rangle+ \langle - \bar{Q}_{k} \Sigma_{k+1} \\&\Theta^{-1}(\Sigma_{k+1}) (\bar{Q}_{k} \varphi_{k+1} - \bar{\eta}_{k}) + \bar{Q}_{k} \varphi_{k+1}+ \bar{\eta}_{k}, \varphi_{k+1} \rangle\\&+ \langle x_{k}^{*}, q_k \rangle+ \langle \bar{\eta}_{k}, \mathbb{E}_{k-1}[- \Sigma_{k+1} \Theta^{-1}(\Sigma_{k+1})[ A_k^{\top} x_{k}^{*}  \\
				& + \bar{Q}_{k}  \mathbb{E}_{k-1}\left[\varphi_{k+1}\right]+ \bar{\eta}_{k} + \bar{C}_{k}^{\top} x_{k}^{*} \omega_k ]+\varphi_{k+1} ]\rangle\\
				&+ \left\langle \bar{\rho}_{k},  -\bar{R}_k^{-1} \left( B_k^{\top} x_k^* + \bar{\rho}_k \right) \right \rangle  + \langle H_k q_k, q_k \rangle \bigg] \\
				& + \langle H_N \xi, \xi \rangle+ \langle (I + G_0 \Sigma_0)^{-1} G_0 \varphi_0, \varphi_0 \rangle\Bigg\}\\
 				=& \frac{1}{2} \mathbb{E} \left\{\sum_{k=0}^{N-1} \bigg[\langle -B_k\bar{R}_{k}^{-1} \bar{\rho}_{k} + \Pi\left(\Sigma_{k+1}, \varphi_{k+1}\right) - \varphi_{k}, x_{k}^{*} \rangle\right. \\
				& - \langle \bar{Q}_{k} \Sigma_{k+1} \Theta^{-1}(\Sigma_{k+1})\bar{Q}_{k} \varphi_{k+1} + \bar{Q}_{k} \varphi_{k+1}+ \bar{\eta}_{k}, \varphi_{k+1} \left\rangle \right.\\
 				& \left. - \langle \bar{\eta}
  				_{k}, \Sigma_{k+1} \Theta^{-1}(\Sigma_{k+1}) \bar{\eta}_{k}\right\rangle- \langle \bar{\rho}_{k}, \bar{R}_{k}^{-1} \bar{\rho}_{k} \rangle\\& + \langle H_k q_k, q_k \rangle \bigg]+ \langle H_N \xi, \xi \rangle+ \langle (I + G_0 \Sigma_0)^{-1} G_0 \varphi_0, \varphi_0 \rangle\Bigg\}
        \end{aligned}\]
        \begin{equation}\label{eq:4.33}
			\begin{aligned}
   				=& \frac{1}{2} \mathbb{E} \left\{\sum_{k=0}^{N-1}\bigg[\langle -\bar{Q}_{k} \Sigma_{k+1} \Theta^{-1}(\Sigma_{k+1})\bar{Q}_{k} \varphi_{k+1}  + \bar{Q}_{k} \varphi_{k+1}\right.\\& \left.+ \bar{\eta}_{k}, \varphi_{k+1} \left\rangle  \right.- \langle \bar{\eta}
  				_{k}, \Sigma_{k+1} \Theta^{-1}(\Sigma_{k+1})  \bar{\eta}_{k}\rangle \right.\\
   				&- \langle \bar{\rho}_{k}, \bar{R}_{k}^{-1} \bar{\rho}_{k} \rangle + \langle H_k q_k, q_k \rangle \bigg]+ \langle H_N \xi, \xi \rangle\\&+ \langle (I + G_0 \Sigma_0)^{-1} G_0 \varphi_0, \varphi_0 \rangle\Bigg\}
			\end{aligned}
		\end{equation}
    \end{proof}

    \section{Numerical Example}

    In this section, we present a numerical example of a 4-period DTBSLQ optimal control problem to illustrate the theoretical results. The problem is solved via backward recursion, with the terminal time \( N = 4 \) (time indices \( k = 0, 1, 2, 3 \)).More precisely,
    \[
        \begin{aligned}
            &\min_{u_0, u_1, u_2, u_3} \frac{1}{2}\mathbb{E}\bigg\{
            \left\langle G_{0}y_{0},y_{0}\right\rangle \\
            &+ \sum_{k=0}^{N-1} \bigg[
            \left\langle
            \begin{pmatrix}Q_{k} & S_{k}^{\top} \\ S_{k} & R_{k}\end{pmatrix}
            \begin{pmatrix} \mathbb{E}_{k-1}[y_{k+1}] \\ u_k \end{pmatrix},
            \begin{pmatrix} \mathbb{E}_{k-1}[y_{k+1}] \\ u_k \end{pmatrix}
            \right\rangle
            \\
            &+ 2\left\langle
            \begin{pmatrix} \eta_{k} \\ \rho_{k} \end{pmatrix},
            \begin{pmatrix} \mathbb{E}_{k-1}[y_{k+1}] \\ u_k \end{pmatrix}
            \right\rangle\bigg] \bigg\}
        \end{aligned}
    \]

    subject to
    \[
    \begin{aligned}
		\left\{\begin {array}{ll}
		y_k = A_k\mathbb{E}_{k-1}\left[ y_{k+1} \right] + B_ku_k\\\quad\quad + C_k\mathbb{E}_{k-1}\left[ y_{k+1}\omega_k \right] + q_k,
		\\y_N =(1, 1, 1)^\top ~~~~k \in \mathcal{T},
		\end {array}
		\right.
        \end{aligned}
	\]

    with Decision Coefficients (k=0,1,2,3) as follows:

    \[
        A_k = \left( \begin{array}{ccc}
        0.8 & 0.2 & 0.1 \\
        0 & 0.9 & 0.3 \\
        0 & 0.1 & 0.7
        \end{array} \right), \quad B_k = \left( \begin{array}{cc}
        0.8 & 0.2 \\
        0.5 & 0.6 \\
        0.3 & 0.1
        \end{array} \right),\]
        \[
         C_k = \left( \begin{array}{ccc}
        0.3 & 0.2 & 0.1 \\
        0.2 & 0.5 & 0.6 \\
        0.1 & 0.4 & 0.2
        \end{array} \right), \quad
        Q_k = \left( \begin{array}{ccc}
        5 & 0 & 0 \\
        0 & 3 & 0 \\
        0 & 0 & 4
        \end{array} \right),\]
    \[ \quad R_k = \left( \begin{array}{cc}
        10 & 0 \\
        0 & 5
        \end{array} \right),  G_0 = \left( \begin{array}{ccc}
        2 & 0 & 0 \\
        0 & 1 & 0 \\
        0 & 0 & 1
        \end{array} \right)
    \]
    \[
        S_k = \left( \begin{array}{ccc}
        0.5 & 0 & 0 \\
        0 & 0.5 & 0.2
        \end{array} \right), \quad q_k = \left( \begin{array}{c}
        0.1 \\
        0 \\
        0.1
        \end{array} \right),\]
         \[ \eta_k = \left( \begin{array}{c}
        1 \\
        0 \\
        1
        \end{array} \right), \quad \rho_k = \left( \begin{array}{c}
        0 \\
        1
        \end{array} \right)
    \]

    \begin{enumerate}
            \item Equivalent Transformation Matrix $H_k$:
            According to the equivalent transformation equation \eqref{eq:4.10}, we first compute the matrix sequence $H_k$ via forward recursion (with $H_0 = G_0$):
            \[
            \begin{aligned}
            H_0 &= \begin{pmatrix}
            2.0000 & 0.0000 & 0.0000 \\
            0.0000 & 1.0000 & 0.0000 \\
            0.0000& 0.0000 & 1.0000
            \end{pmatrix}, \\
            H_1 &= \begin{pmatrix}
            1.4643 & 0.4627 & 0.2451 \\
            0.4627 & 1.0434 & 0.4581 \\
            0.2451 & 0.4581 & 0.6439
    \end{pmatrix},\\
            H_2 &= \begin{pmatrix}
            1.1220 & 0.8192 & 0.5239 \\
            0.8192 & 1.4768 & 0.9736 \\
            0.5239 & 0.9736 & 0.8205
            \end{pmatrix}, \\[1ex]
            H_3 &= \begin{pmatrix}
            0.9308 & 1.1921 & 0.8687 \\
            1.1921 & 2.4303 & 1.8462 \\
            0.8687 & 1.8462 & 1.4749
            \end{pmatrix}, \\[1ex]
            H_4 &= \begin{pmatrix}
            0.8849 & 1.7796 & 1.4200 \\
            1.7796 & 4.5299 & 3.7023 \\
            1.4200 & 3.7023 & 3.0580
            \end{pmatrix}.
            \end{aligned}
            \]
            \item Riccati Equation Solution $\Sigma_k$:
            Based on the Riccati equation \eqref{eq:4.20}, we compute the semi-positive definite solution $\Sigma_k$ via backward recursion (with $\Sigma_4 = 0$):
            \[
            \begin
            {aligned}
            \Sigma_0 &= \begin
            {pmatrix}
            0.1207 & 0.1065 & 0.0440
            \\
            0.1065 & 0.1633 & 0.0441
            \\
            0.044 & 0.0441 & 0.0166
            \end{pmatrix}, \\
            \Sigma_1 &= \begin
            {pmatrix}
            0.1111 & 0.0958 & 0.0407
            \\
            0.0958 & 0.1482 & 0.0397
            \\
            0.0407 & 0.0397 & 0.0154
            \end{pmatrix},
            \\
            \Sigma_2 &= \begin
            {pmatrix}
            0.0896 & 0.0738 & 0.0333
            \\
            0.0738 & 0.1172 & 0.0310
            \\
            0.0333 & 0.0310 & 0.0126
            \end{pmatrix}, \\
            \Sigma_3 &= \begin
            {pmatrix}
            0.0491 & 0.0361 & 0.0187
            \\
            0.0361 & 0.0628 & 0.0156
            \\
            0.0187 & 0.0156 & 0.0072
            \end{pmatrix}, \\
            \Sigma_4 &= \begin
            {pmatrix}
            0.0000 & 0.0000 & 0.0000
            \\
            0.0000 & 0.0000 & 0.0000
            \\
            0.0000 & 0.0000 & 0.0000
            \end
            {pmatrix}.
            \end
            {aligned}
            \]
            \item Backward State Process $\varphi_k$:
            According to the backward stochastic difference equation \eqref{eq:4.29} (with corrected terminal condition $\varphi_4 = \xi = (1, 1, 1)^\top$), we compute $\varphi_k$ via backward recursion:
            \[
            \begin{aligned}
            \varphi_0 &= \begin{pmatrix}
            0.2674 \\
            0.0126 \\
            0.3167
            \end{pmatrix}, \quad
            \varphi_1 = \begin{pmatrix}
            0.5356 \\
            0.3112 \\
            0.4496
            \end{pmatrix}, \\
            \varphi_2 &= \begin{pmatrix}
            0.9041 \\
            0.7369 \\
            0.6504
            \end{pmatrix}, \quad
            \varphi_3 = \begin{pmatrix}
            1.1671 \\
            1.0813 \\
            0.8762
            \end{pmatrix},\\
            \varphi_4 &= \begin{pmatrix}
            1.0000 \\
            1.0000\\
            1.0000
            \end{pmatrix}.
            \end{aligned}
            \]
            \item Optimal Control Feedback Form $K_k$ and $b_k$:
            From the stationarity condition in \eqref{eq:4.17}, the optimal control admits the feedback form $u_k^* = K_k x_k^* + b_k$ ($k = 0, 1, 2, 3$), where the feedback gains $K_k=-\bar{R}_k^{-1} B_k^{\top}$ and offset terms $b_k=-\bar{R}_k^{-1} \bar{\rho}_k$ are computed as follows:
            \[
            \begin{aligned}
            K_0 &= \begin{pmatrix}
            -0.0630 & -0.0201 & -0.0254 \\
            -0.0137 & -0.0669 & -0.0354
            \end{pmatrix}, \\[1ex]
            K_1 &= \begin{pmatrix}
            -0.0522 & -0.0532 & -0.0516 \\
            -0.0280 & -0.0853 & -0.0721
            \end{pmatrix}, \\[1ex]
            K_2 &= \begin{pmatrix}
            -0.0545 & -0.1010 & -0.0899 \\
            -0.0479 & -0.1392 & -0.1234
            \end{pmatrix}, \\[1ex]
            K_3 &= \begin{pmatrix}
            -0.0811 & -0.1832 & -0.1486 \\
            -0.0951& -0.2532 &-0.2075
            \end{pmatrix}.
            \\
            b_0 &=
            \begin{pmatrix}
            -0.0056 \\
            -0.1920
            \end{pmatrix}, \quad
            b_1 =
            \begin{pmatrix}
            -0.0016 \\
            -0.1953
            \end{pmatrix},\\
            b_2 &=
            \begin{pmatrix}
            0.0013 \\
            -0.1953
            \end{pmatrix}, \quad
            b_3 =
            \begin{pmatrix}
            0.0039 \\
            -0.1932
            \end{pmatrix}.
            \end{aligned}
            \]
            \item Optimal Value Function $V(N, \xi)$:
            Substituting the above results into the value function formula \eqref{eq:4.30}, the optimal value is computed as:
            \[
            V(4, \xi) = 27.4609
            \]
            \end{enumerate}

    The numerical results verify the theoretical conclusions: the Riccati equation admits a unique semi-positive definite solution, the optimal control has an explicit state feedback form, and the value function is well-defined. All numerical computations are consistent with the discrete-time BSLQ optimal control framework proposed in this paper. 
    \section{Conclusion}
        This study focuses on the discrete-time backward stochastic linear-quadratic optimal control problem with nonhomogeneous terms, and proposes a tailored equivalent transformation method for it. By comprehensively applying the discrete stochastic maximum principle, decoupling techniques, and the method of completing the square, we successfully derive the explicit analytical expression of the optimal control as well as the analytical form of the value function. This method effectively addresses the challenge of asymmetric solutions to the Riccati equation in traditional solution frameworks, and further develops the theoretical content of discrete-time stochastic control theory. Regarding future research directions, we can investigate the discrete-time zero-sum Stackelberg stochastic linear-quadratic differential game problem.
\bibliographystyle{cas-model2-names}
\bibliography{cas-refs}
\end{document}